\documentclass[a4paper,reqno]{amsart}
\usepackage[utf8]{inputenc}
\usepackage[T1]{fontenc}
\usepackage[british]{babel}
\usepackage{amsmath,amssymb}
\usepackage{graphicx}
\usepackage[all]{xy}
\usepackage{extarrows}
\usepackage{mathbbol}
\usepackage{mathabx}
\usepackage{enumerate}

\newdir{>>}{{}*!/3.5pt/:(1,-.2)@^{>}*!/3.5pt/:(1,+.2)@_{>}*!/7pt/:(1,-.2)@^{>}*!/7pt/:(1,+.2)@_{>}}
\newdir{ >>}{{}*!/8pt/@{|}*!/3.5pt/:(1,-.2)@^{>}*!/3.5pt/:(1,+.2)@_{>}}
\newdir{ |>}{{}*!/-3.5pt/@{|}*!/-8pt/:(1,-.2)@^{>}*!/-8pt/:(1,+.2)@_{>}}
\newdir{ >}{{}*!/-8pt/@{>}}
\newdir{>}{{}*:(1,-.2)@^{>}*:(1,+.2)@_{>}}
\newdir{<}{{}*:(1,+.2)@^{<}*:(1,-.2)@_{<}}

\theoremstyle{plain}

\newtheorem{theorem}[subsection]{Theorem}
\newtheorem{lemma}[subsection]{Lemma}
\newtheorem{proposition}[subsection]{Proposition}
\newtheorem{corollary}[subsection]{Corollary}

\theoremstyle{definition}
\newtheorem{definition}[subsection]{Definition}
\newtheorem{remark}[subsection]{Remark}

\newtheorem{example}[subsection]{Example}

\newtheorem{construction}[subsection]{Construction}

\newenvironment{tfae}
{
\begin{enumerate}}
{\end{enumerate}}

\numberwithin{figure}{section}

\newcommand{\pushoutcorner}[1][dr]{\save*!/#1+1.2pc/#1:(1,-1)@^{|-}\restore}
\newcommand{\pullbackcorner}[1][dr]{\save*!/#1-1.2pc/#1:(-1,1)@^{|-}\restore}

\renewcommand{\binom}{\genfrac{\lgroup}{\rgroup}{0pt}{1}}
\newcommand{\pushout}{\binom}
\newcommand{\morph}[2]{\bigl(\begin{smallmatrix}
	#1 \\ #2
\end{smallmatrix}\bigr)}

\renewcommand{\iff}{\Leftrightarrow}
\newcommand{\comp}{\raisebox{0.2mm}{\ensuremath{\scriptstyle{\circ}}}}
\newcommand{\LieAlg}{\mathbf{Lie}_R}
\newcommand{\Lie}{\mathit{Lie}}

\newcommand{\Coeq}{\mathit{Coeq}}
\newcommand{\Grp}{\mathbf{Grp}}
\newcommand{\PreXMod}{\mathbf{PreXMod}}
\newcommand{\XMod}{\mathbf{XMod}}

\newcommand{\Pt}{\mathbf{Pt}}
\newcommand{\Act}{\mathbf{Act}}

\newcommand{\Nil}{\mathbf{Nil}}
\newcommand{\nil}{\mathit{Nil}}

\newcommand{\Grpd}{\mathbf{Grpd}}
\newcommand{\RG}{\mathbf{RG}}
\newcommand{\XSqr}{\mathbf{XSqr}}
\newcommand{\cat}{\mathbf{cat}}
\newcommand{\A}{\mathbb{A}}

\newcommand{\Z}{\mathbb{Z}}

\newcommand{\Qq}{\mathcal{Q}}

\newcommand{\SH}{{\rm (SH)}}
\newcommand{\X}[1]{{\rm (X.#1)}}
\newcommand{\W}[1]{{\rm (W.#1)}}

\title[The non-abelian tensor product]{An intrinsic approach to the non-abelian tensor product via internal crossed squares}

\author{Davide di Micco}
\address[Davide di Micco]{Università degli Studi di Milano, Via Saldini 50, 20133 Milano, Italy}
\email[Davide di Micco]{davide.dimicco@unimi.it}

\author{Tim Van~der Linden}
\address[Tim Van~der Linden]{Institut de
Recherche en Math\'ematique et Physique, Universit\'e catholique
de Louvain, che\-min du cyclotron~2 bte~L7.01.02, B--1348
Louvain-la-Neuve, Belgium}
\email[Tim Van~der Linden]{tim.vanderlinden@uclouvain.be}

\thanks{The second author is a Research
Associate of the Fonds de la Recherche Scientifique--FNRS}

\keywords{Semi-abelian category, pair of compatible actions, internal action, crossed module, crossed square, commutator, non-abelian tensor product}

\subjclass[2020]{18D40, 18E13, 20J15}

\begin{document}

\begin{abstract}
We explain how, in the context of a semi-abelian category, the concept of an \emph{internal crossed square} may be used to set up an intrinsic approach to the Brown-Loday \emph{non-abelian tensor product}.
\end{abstract}

\maketitle

\section{Introduction}
The aim of this article is to explain how, in the context of a semi-abelian category~\cite{JMT02,BB04}, \emph{internal crossed squares} can be used to set up an intrinsic approach to the \emph{non-abelian tensor product}. Both concepts were originally introduced for groups (by Guin-Wal\'ery, Brown and Loday, in \cite{GWL80,Lod82,BL87}) and for Lie algebras (by Ellis, in \cite{Ell91}). 

Recall that a \emph{crossed module} is a group homomorphism $\partial\colon M\to L$ together with an action of $L$ on $M$, satisfying suitable compatibility conditions. The category $\XMod$ of crossed modules is equivalent to the category $\Grpd(\Grp)$ of internal groupoids in the category of groups, via the following constructions. The normalisation functor sends an internal groupoid (of which we only depict the underlying reflexive graph)
\begin{align*}
\xymatrix{
X \ar@<6pt>[r]^-{d} \ar@<-4pt>[r]_-{c} & L \ar@<-1pt>[l]|-{e}
}
\end{align*}
to the map $c\comp k_d\colon K_d\to L$ (where $k_d\colon {K_d}\to X$ denotes the kernel of $d$), together with the action $\xi$ determined by conjugation in $X$. Its pseudo-inverse sends a crossed module $(\partial\colon M\to L,\xi)$ to the semidirect product $M\rtimes_\xi L$ with the suitable projections on~$L$.

A \emph{crossed square} (of groups) is a \emph{two-dimensional crossed module}, in the following precise sense. The internal groupoid construction may be repeated, which gives us the category $\Grpd^2(\Grp)=\Grpd(\Grpd(\Grp))$ of internal \emph{double} groupoids in~$\Grp$. Given such an internal double groupoid as in Figure~\ref{Star}, 
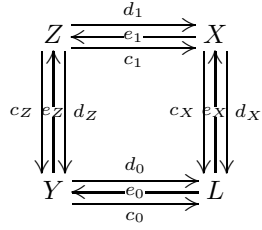
\begin{figure}
$\vcenter{
\xymatrix@=2em{
Z \ar@<1ex>[dd]^-{d_Z} \ar@<-1ex>[dd]_-{c_Z} \ar@<1ex>[rr]^-{d_1} \ar@<-1ex>[rr]_-{c_1} && X \ar[ll]|-{e_1} \ar@<1ex>[dd]^-{d_X} \ar@<-1ex>[dd]_-{c_X}\\
\\
Y \ar[uu]|-{e_Z} \ar@<1ex>[rr]^-{d_0} \ar@<-1ex>[rr]_-{c_0} && L \ar[uu]|-{e_X} \ar[ll]|-{e_0}
}
}$
\caption{An internal double groupoid, viewed as a double reflexive graph}\label{Star}
\end{figure}
viewed as a diagram in $\Grp$ (in which again the composition maps are omitted), we may take the normalisation functor vertically and horizontally to obtain a commutative square
\begin{equation}\label{Diag:Sq}
\vcenter{\xymatrix{P \ar[r]^{p_M} \ar[d]_{p_N} & M\ar[d]^\mu\\
N \ar[r]_{\nu} & L.}}
\end{equation}
The given double groupoid structure naturally induces actions of~$L$ on~$M$, $P$ and~$N$, of $M$ and~$N$ on~$P$, etc. One may now ask, whether it is possible to equip a given commutative square of group homomorphisms with suitable actions (and, possibly, additional maps), in such a way that an internal double groupoid may be recovered---thus extending the equivalence $\XMod\simeq \Grpd(\Grp)$ in order to capture double groupoids in~$\Grp$ as commutative squares with extra structure. The concept of a crossed square (\cite{GWL80,Lod82,BL87}, see Definition~\ref{defi:explicit defi of crossed square of groups} below) answers this question, and does indeed give rise to a category equivalence $\XSqr\simeq \Grpd^2(\Grp)$.

\emph{Internal crossed squares} answer the same question, now asked for a general base category $\A$, which we take to be semi-abelian (in the sense of~\cite{JMT02}; for the sake of simplicity, we ask our categories to satisfy a relatively mild extra condition called the \emph{Smith is Huq} condition \SH\ in~\cite{MFVdL12}---see Section~\ref{Sec:Actions} and Section~\ref{Sec:Xmod} for detailed explanations). The work of Janelidze~\cite{Jan03} provides an explicit description of internal crossed modules in~$\A$, together with an equivalence of categories $\XMod(\A)\simeq \Grpd(\A)$ which reduces to $\XMod\simeq\Grpd(\Grp)$ when $\A=\Grp$. Since the category of internal groupoids in a semi-abelian category is again semi-abelian~\cite{Bourn-Gran}, the category of internal crossed modules is semi-abelian as well. Hence this construction may be repeated as in~\cite{EG-honfg}, and thus we obtain an equivalence $\XMod^2(\A)\simeq \Grpd^2(\A)$. We may now write $\XSqr(\A)= \XMod^2(\A)$ and say that a \emph{crossed square in $\A$} is an internal crossed module of internal crossed modules in~$\A$. Indeed, any such double internal crossed module has an underlying commutative square in~$\A$, which the crossed module structures equip with suitable internal actions in such a way that an internal double groupoid may be recovered. The internal action structure is, however, far from being transparent, and thus merits further explicitation.

Yet, we shall see that even this tentative and very abstract general definition is concrete enough to serve as a basis for an intrinsic approach to the \emph{non-abelian tensor product}. Originally this tensor product (of two groups $M$ and $N$ acting on each other in a certain \lq\lq compatible\rq\rq\ way) was defined in~\cite{BL87} via a presentation in terms of generators and relations. In the article \cite{dMVdL19} we investigated how to extend the concept of a pair of compatible actions (of two given objects $M$ and $N$ acting on each other) to the semi-abelian setting. A~key feature (known already for groups and Lie algebras) of such a pair of compatible actions is that it is equivalent to the datum of a third object $L$ and two internal crossed module structures $\mu\colon M\to L$ and $\nu\colon N\to L$. According to another result of Brown and Loday~\cite{BL87}, given two $L$-crossed modules $\mu$ and $\nu$, the crossed module $\mu\comp p_M=\nu\comp p_N\colon P\to L$ in a crossed square of groups of the form~\eqref{Diag:Sq} happens to be the tensor product of $M$ and $N$ with respect to the actions of $M$ and $N$ on each other, induced by the crossed module structures of $\mu\colon M\to L$ and $\nu\colon N\to L$, if and only if the crossed square is the initial object in the category of all crossed squares over the given crossed modules $\mu $ and $\nu$. This property of course determines the tensor product, and it may actually be taken as a definition.

Concretely this means that in a semi-abelian category (satisfying the condition \SH), the non-abelian tensor product of two objects acting compatibly on one another may be constructed as follows.
\begin{enumerate}
\item Consider the internal $L$-crossed modules $\mu\colon M\to L$ and $\nu\colon N\to L$ corresponding to the given actions.
\item Use the equivalence $\XMod(\A)\simeq \Grpd(\A)$ to obtain internal groupoids
\begin{equation*}
\vcenter{
\xymatrix{
Y \ar@<1ex>[rr]^-{d_0} \ar@<-1ex>[rr]_-{c_0} && L \ar[rr]|-{e_X} \ar[ll]|-{e_0} && X. \ar@<1ex>[ll]^-{d_X} \ar@<-1ex>[ll]_-{c_X}
}
}
\end{equation*}
\item Take the pushout of $e_0$ and $e_X$ to find the double reflexive graph in Figure~\ref{Star}.
\item This double reflexive graph is not yet an internal double groupoid; reflect it into $\Grpd^2(\A)$ by taking the quotient of $Z$ by the join of commutators $[K_{c_Z},K_{d_Z}]\vee [K_{c_1},K_{d_1}]$.
\item The resulting internal double groupoid normalises to an internal crossed square 
\[
\xymatrix{M\otimes N \ar[r]^-{p_M} \ar[d]_-{p_N} & M\ar[d]^-\mu\\
N \ar[r]_-{\nu} & L,}
\]
whose structure involves a crossed module $M\otimes N\to L$. By definition, this is the non-abelian tensor product of the given pair of compatible actions.
\end{enumerate}
By known properties of the non-abelian tensor product for groups and Lie algebras, this reduces to the classical definitions in those cases (Proposition~\ref{prop:prop2.15 in BL87} and Proposition~\ref{Tensor for Lie algebras}). In Section~\ref{Sec:Examples} we give further concrete information in the case of a pair of inclusions of normal subobjects (where we obtain a crossed module whose image is a commutator) and the case of a pair of abelian objects acting trivially upon one another (then we regain the \emph{bilinear product} of \cite{HVdL2}).

In the forthcoming article \cite{dMVdL19.4}, we use this general version of the non-abelian tensor product to prove a result on the existence of universal central extensions of internal crossed modules over a fixed base object. Our present article is devoted to exploring some basic properties of the definition, and showing that in certain cases, the tensor product may be used to give an explicit description of an object of $\XSqr(\A)$ as a square~\eqref{Diag:Sq} in $\A$ equipped with suitable actions and a morphism $h\colon M\otimes N\to P$. This extends the explicit descriptions for groups and Lie algebras to the general setting. It is, however, not yet clear to us whether this description is always valid---see Section~\ref{section:towards crossed squares through the non-abelian tensor product}.

We start by recalling basic notions and techniques of commutators and internal actions (Section~\ref{Sec:Actions}) and crossed modules (Section~\ref{Sec:Xmod}) in semi-abelian categories. Section~\ref{Section:Double groupoids and double reflexive graphs} discusses double reflexive graphs and internal double groupoids, and Section~\ref{Section:crossed squares} is devoted to the basic theory of crossed squares. In Section~\ref{Section Non-Abelian Tensor} we explain how this may be used in an intrinsic approach to the non-abelian tensor product. Section~\ref{Sec:Examples} gives examples. Section~\ref{section:towards crossed squares through the non-abelian tensor product} treats a (partial) description of crossed squares in terms of the tensor product.

\section{Commutators and internal actions}\label{Sec:Actions}
Here we recall basic properties of commutators and internal actions, needed in what follows. We start with the equivalence between internal actions and split extensions. 

A \emph{point} $(p,s)$ in a category $\A$ is a split epimorphism $p\colon A\to B$ together with a chosen splitting $s\colon B\to A$, so that $p\comp s=1_B$. The category $\Pt(\A)$ of \emph{points in~$\A$} has, as objects, points in $\A$, and as morphisms, natural transformations between such (as in Lemma~\ref{lemma:epimorphism of points is pushout}). 

If $\A$ is a semi-abelian category, then a point $(p,s)$ with a chosen kernel $k$ of $p$ is the same thing as a \emph{split extension} in $\A$: a split short exact sequence
\[
\xymatrix{0 \ar[r] & K \ar@{{ |>}->}[r]^k & A \ar@{-{ >>}}@<.5ex>[r]^-p & B \ar@<.5ex>[l]^-s \ar[r] & 0,}
\]
which means that $k$ is the kernel of $p$, $p$ is the cokernel of $k$, and $p\comp s=1_B$. In such a split extension, $k$ and $s$ are jointly extremal-epimorphic.

Via a semi-direct product construction~\cite{BJ98}, we have an equivalence $\Pt(\A)\simeq\Act(\A)$, where the latter category of \emph{(internal) actions} in $\A$ consists of the algebras of the monad $(A\flat (-),\eta^A,\mu^A)$ defined through 
\begin{equation}\label{Def Flat}
\xymatrix{0 \ar[r]& A\flat B\ar@{{ |>}->}[r]^-{k_{A,B}} & A+B \ar@<.5ex>@{-{ >>}}[r]^-{\binom{1_A}{0}} & A\ar[r] \ar@<.5ex>[l]^-{\iota_A} & 0.}	
\end{equation}
One of the functors in the equivalence sends a point $(p,s)$ to the action $(B,K_p,\xi)$ in
\begin{align*}
\xymatrix{
0 \ar[r] & B\flat K_p \ar@{.>}[d]_-{\xi} \ar@{{ |>}->}[r]^-{k_{B,K_p}} & B+K_p \ar[d]^-{\binom{s}{k_p}} \ar@{-{ >>}}[r]^-{\binom{1_B}{0}} & B\ar@{=}[d] \ar[r] & 0\\
0 \ar[r] & K_p \ar@{{ |>}->}[r]_-{k_p} & A \ar@{-{ >>}}[r]_-{p} & B\ar[r] & 0.
}
\end{align*}
The other functor sends an action $(A,X,\xi)$ to the induced semidirect product, which is the point $({\pi_\xi\colon X\rtimes_{\xi} A \to A}, {i_\xi\colon A\to X\rtimes_{\xi} A })$, where $X\rtimes_{\xi} A$ is the coequaliser
\[
\xymatrixcolsep{4pc}
\xymatrixrowsep{3pc}
\xymatrix{
A\flat X \ar@<4pt>[r]^-{i_X\circ\xi} \ar@<-1pt>[r]_-{k_{A,X}} & A+X \ar[r]^{\sigma_{\xi}} & X\rtimes_{\xi} A,
}
\]
the map $\pi_\xi\colon X\rtimes_{\xi} A\to A$ is the unique map such that $\binom{1_A}{0}=\pi_{\xi}\comp \sigma_{\xi}$, and finally $i_\xi=\sigma_\xi\comp i_A$. 

We will denote $X\rtimes_{\xi} A$ as $X\rtimes A$ if there is no risk of confusion regarding the action involved. The map $k\coloneq\sigma_{\xi}\comp i_X\colon X\to X\rtimes_{\xi} A $ is always the kernel of $\pi_{\xi}$: it is easy to see that $\pi_{\xi}\comp k=0$, whereas for the universal property some work needs to be done. 

\begin{example}
\label{ex:semidirectproduct induced by trivial action}
The \emph{trivial action} $(A,X,\tau^A_X)$ is $\tau^A_X=\binom{0}{1_X}\comp k_{A,X}\colon A\flat X\to X$. We have 
\begin{align*}
(X\rtimes_{\tau^A_X} A,\sigma_{\tau^A_X})& \cong \Coeq(i_X\comp(\binom{0}{1_X}\comp k_{A,X}),k_{A,X}).
\end{align*}
Both $\binom{1_A}{0}$ and $\binom{0}{1_X}$ coequalise the two maps, and it is indeed not hard to see that 
\[
\Coeq(i_X\comp\binom{0}{1_X}\comp k_{A,X},k_{A,X})\cong (A\times X,\langle\binom{1_A}{0},\binom{0}{1_X}\rangle\colon A+X\to A\times X). 
\]
\end{example}

\begin{example}
For each object $A$ we can define the conjugation action $(A,A,\chi_A)$ through $\chi_A=\binom{1_A}{1_A}\comp k_{A,A}\colon A\flat A\to A$. Then we have that 
\begin{align*}
(A\rtimes_{\chi_A} A,\sigma_{\chi_A})& \cong \Coeq(i_2\comp(\binom{1_A}{1_A}\comp k_{A,A}),k_{A,A})\\&\cong (A\times A,\langle\binom{1_A}{0},\binom{1_A}{1_A}\rangle\colon A+A\to A\times A).
\end{align*}
\end{example}

\begin{remark}
\label{rmk:acting is like conjugating in the semidirect product}
From the definition, it follows that the square on the left
\[
\xymatrix{
A\flat X \ar@{{ |>}->}[r]^-{k_{A,X}} \ar[d]_-{\xi} & A+X \ar[d]^-{\sigma_{\xi}}\\
X \ar@{{ |>}->}[r]_-{k_{\pi_{\xi}}} & X\rtimes_{\xi}A
}
\qquad\qquad
\xymatrix{
A\flat X \ar[d]_-{i_{\xi}\flat k_{\pi_{\xi}}} \ar[r]^-{\xi} & X \ar[d]^-{k_{\pi_{\xi}}}\\
(X\rtimes_{\xi}A)\flat(X\rtimes_{\xi}A) \ar[r]_-{\chi_{(X\rtimes_{\xi}A)}} & X\rtimes_{\xi}A
}
\]
is both a pushout and a pullback. In fact, also the square on the right commutes, which means that \lq\lq computing an action\rq\rq\ is the same as \lq\lq computing the conjugation in the induced semidirect product\rq\rq.
\end{remark}

\begin{lemma}
\label{lemma:epimorphism of points is pushout}
Consider a morphism of points $(f,g)\colon(p_0,s_0)\to (p_1,s_1)$
\[
\xymatrix{
A_0 \ar[d]_-{f} \ar@<.5ex>[r]^-{p_0} & B_0 \ar@<.5ex>[l]^-{s_0} \ar[d]^-{g}\\
A_1 \ar@<.5ex>[r]^-{p_1} & B_1. \ar@<.5ex>[l]^-{s_1}
}
\]
If $f$ is an epimorphism then the right pointing square is a pushout. Dually, if~$f$ is a monomorphism, then the left pointing square is a pullback.
\end{lemma}

\subsection{Basic commutator theory}
We turn to the definition and stability properties of binary and ternary Higgins commutators.

\begin{definition}[\cite{MM10, CJ03, HVdL11}]
\label{defi:cosmash products}
Given two objects $A$ and $B$ in $\A$, the morphism
\[
\Sigma_{A,B}\coloneq\binom{\langle 1_A,0\rangle}{\langle0,1_B\rangle}=\langle \binom{1_A}{0}\binom{0}{1_B}\rangle\colon A+B \longrightarrow A\times B
\]
is a regular epimorphism. Hence taking its kernel we find the short exact sequence
\[
\xymatrixcolsep{3pc}
\xymatrix{
0 \ar[r] & A\diamond B \ar@{{ >}->}[r]^-{h_{A,B}} & A+B \ar[r]^-{\Sigma_{A,B}} & A\times B \ar[r] & 0
}
\]
where $A\diamond B$ is called the \emph{cosmash product} of $A$ and $B$.
\end{definition}

\begin{definition}[\cite{Hig56,MM10}]
Given two subobjects $(M,m)$ and $(N,n)$ of an object $X$, we define their \emph{Higgins commutator} as the image of the map $\binom{m}{n}\comp h_{M,N}$, that is the subobject of $X$ given by the factorisation
\[
\xymatrix{
M\diamond N \ar@{.{ >>}}[d] \ar@{{ |>}->}[r]^-{h_{M,N}} & M+N \ar[d]^-{\binom{m}{n}}\\
[M,N] \ar@{{ >}.>}[r] & X.
}
\]
\end{definition}

The \emph{Huq commutator} ${[M,N]}^{\Qq}_{X}$ of $M$ and $N$ can be seen as the normal closure in~$X$ of their Higgins commutator. Note that one vanishes if and only if so does the other. An object $X$ is said to be \emph{abelian} when $[X,X]$ is trivial; this happens precisely when $X$ admits a (necessarily unique) internal abelian group structure---see~\cite{BB04}. 

\begin{definition}[\cite{CJ03,HVdL11,HL13}]
Given three objects $A$, $B$ and $C$ in $\A$, consider the map 
\[
\Sigma_{A,B,C}=
\begin{pmatrix}
i_A & i_A & 0 \\
i_B & 0 & i_B \\
0 & i_C & i_C
\end{pmatrix}
\colon A+B+C \longrightarrow (A+B)\times(A+C)\times(B+C)
\]
and its kernel $h_{A,B,C}\colon A\diamond B\diamond C \to A+B+C$. The object $A\diamond B\diamond C$ is called the \emph{cosmash product} of $A$, $B$ and $C$. 

Given three subobjects $(K,k)$, $(M,m)$ and $(N,n)$ of an object $X$, we define their \emph{Higgins commutator} as the subobject of $X$ given by the factorisation 
\[
\xymatrix{
K\diamond M\diamond N \ar@{.{ >>}}[d] \ar@{{ |>}->}[r]^-{h_{K,M,N}} & K+M+N \ar[d]^-{\Bigl\lgroup\begin{smallmatrix}
	k \\ m\\ n
\end{smallmatrix}\Bigr\rgroup}\\
[K,M,N] \ar@{{ >}.>}[r] & X.
}
\]
We call $[K,M,N]$ the \emph{ternary Higgins commutator} of $K$, $M$ and $N$ in $X$.
\end{definition}

\begin{proposition}[\cite{HVdL11, HL13}]\label{Higgins properties}
Suppose $K_1$, $K_2$, $K_3\leq X$. Then we have the following (in)equalities of subobjects of $X$:
\begin{enumerate}
\item[0.] if $K_1=0$ then $[K_1,K_2]=0=[K_1,K_2,K_3]$;
\item $[K_1,K_2]=[K_2,K_1]$ and for $\sigma\in S_3$, $[K_1,K_2,K_3]=[K_{\sigma(1)},K_{\sigma(2)},K_{\sigma(3)}]$;
\item for $f\colon {X\to Y}$ any regular epimorphism, $f[K_1,K_2 ]=[f(K_1),f(K_2)]\leq Y$ and $f[K_1,K_2,K_3 ]=[f(K_1),f(K_2),f(K_3)]\leq Y$;
\item $[L_1,K_2]\leq [K_1,K_2]$ and $[L_1,K_2,K_3]\leq [K_1,K_2,K_3]$ when $L_1\leq K_1$;
\item $[[K_1,K_{2}],K_{3}]\leq [K_1,K_2,K_{3}]$;
\item $[K_1,K_1,K_2]\leq [K_1,K_2]$;
\item $[K_1,K_{2}\vee K_3]=[K_1,K_2]\vee [K_1,K_3]\vee [K_1,K_{2},K_3]$.
\end{enumerate}
\end{proposition}

A semi-abelian category is said to satisfy the \emph{Smith is Huq condition} \SH\ when the Smith-Pedicchio commutator~\cite{Ped95b} of two internal equivalence relations vanishes if and only if so does the Huq commutator of their associated normal subobjects~\cite{BB04,MFVdL12}. As explained in~\cite{HVdL11}, in terms of Higgins commutators, this amounts to the condition that whenever $M$, $N\lhd L$ are normal subobjects, $[M,N]=0$ implies $[M,N,L]=0$. As a consequence, under \SH, Higgins commutators suffice for the description of internal groupoids. Furthermore, the characterisation of internal crossed modules given in~\cite{Jan03} simplifies---see below. From now on, unless mentioned otherwise, this is the context we shall work in. Examples of semi-abelian categories that satisfy \SH\ include the categories of groups, (commutative) rings (not necessarily unitary), Lie algebras over a commutative ring with unit, Poisson algebras and associative algebras, as are all varieties of such algebras, and crossed modules over those. In fact, all \emph{Orzech categories of interest}~\cite{Orz72,CGVdL15b} are examples. On the other hand, the category of loops is semi-abelian but does not satisfy \SH.

\section{Internal reflexive graphs, groupoids and (pre-)crossed modules}\label{Sec:Xmod}
Here we recall how to characterise internal reflexive graphs and internal groupoids as internal precrossed modules and internal crossed modules in a semi-abelian category $\A$ that satisfies the \emph{Smith is Huq condition} \SH. When the context is clear, we sometimes drop the adjective \emph{internal}.
\begin{definition}
A \emph{reflexive graph} $(C_1,C_0,d,c,e)$ in $\A$ is given by a diagram
\begin{align*}
\xymatrix{
C_1 \ar@<6pt>[r]^-{d} \ar@<-4pt>[r]_-{c} & C_0 \ar@<-1pt>[l]|-{e}
}
\end{align*}
such that $d\comp e=1_{C_0}=c\comp e$. A morphism of reflexive graphs is a natural transformation between two such diagrams. This determines a category $\RG(\A)$.
\end{definition}

\begin{lemma}[\cite{MFVdL12}] 
\label{lemma: mult struc <=> HUQ=0}
Let $\A$ be a semi-abelian category with \SH. Given a reflexive graph $(C_1,C_0,d,c,e)$, it admits a (unique) internal groupoid structure if and only if ${[K_{d},K_{c}]}=0$.
\end{lemma}

The forgetful functor $\Grpd(\A)\to \RG(\A)$ has a left adjoint ${\RG(\A)\to \Grpd(\A)}$. The image of a reflexive graph $(C_1,C_0,d,c,e)$ through this functor is the reflexive graph $(C'_1,C'_0,d',c',e')$ where $C'_0=C_0$, $C'_1={C_1}/{{[K_{d},K_{c}]}^{\Qq}_{C_1}}$ and $d'$, $c'$, $e'$ are induced by $d$, $c$ and $e$, respectively.
By Lemma~\ref{lemma: mult struc <=> HUQ=0} this is indeed an internal groupoid.

\begin{definition}[\cite{Jan03,MM10.2}]
\label{defi:definition of internal pre-crossed modules}
An \emph{internal pre-crossed module} $(X\xrightarrow{\partial}A,\xi)$ in a semi-abelian category $\A$ with \SH\ is given by an internal action $(A,X,\xi)$ and a morphism $\partial\colon X\to A$ such that the diagram
\[
\vcenter{
\xymatrix{
A\flat X\ar[d]_-{1_A\flat\partial} \ar[r]^-{\xi} & X\ar[d]^-{\partial} \\
A\flat A \ar[r]_-{\chi_A} & A
}
}
\]
commutes. We write $\PreXMod(\A)$ for the category of precrossed modules with the suitable morphisms between them, which are morphisms of arrows that preserve the action.
\end{definition}

\begin{construction}
\label{constr:particular point associated to a crossed module}
By using the correspondence between $\Pt(\A)$ and $\Act(\A)$ we can map each internal pre-crossed module to a particular reflexive graph
\begin{align*}
\xymatrix{
X \ar@{{ |>}->}[rr]^-{k_d} && X\rtimes_{\xi} A \ar@<1ex>[rr]^-{d} \ar@<-1ex>[rr]_-{c} && A, \ar[ll]|-{e}
}
\end{align*}
where $c\comp e=1_A=d\comp e$. Details of this construction are as follows: first we obtain $X\rtimes_{\xi} A$ and the maps $d$, $e$ and $k_d$ by computing the point associated to the action $\xi$. Then we define the map $c$, so that $c\comp\sigma_\xi=\binom{1_A}{\partial}\colon A+X\to A$. Notice that $\binom{1_A}{\partial}\comp(i_X\comp \xi)=\binom{1_A}{\partial}\comp k_{A,X}$ due to the fact that $(A,X,\xi,\partial)$ is a pre-crossed module. Finally we deduce that $c\comp k=\partial$ and that $c\comp e=1_A$. This determines a category equivalence between internal reflexive graphs and internal pre-crossed modules.
\end{construction}

\begin{definition}[\cite{Jan03,MM10.2,HVdL11}]
An \emph{internal crossed module} in a semi-abelian category $\A$ with \SH\ is an internal pre-crossed module $(X\xrightarrow{\partial}A,\xi)$ satisfying the so-called \emph{Peiffer condition}, which is the commutativity of the diagram
\begin{align*}
\xymatrix{
X\flat X\ar[d]_-{\partial\flat 1_X} \ar[r]^-{\chi_X} & X\ar@{=}[d]\\
A\flat X \ar[r]_-{\xi} & X.
}
\end{align*}
\end{definition}
As follows from the results of \cite{Jan03,MM10.2,HVdL11}, the equivalence $\PreXMod(\A)\simeq\RG(\A)$ restricts to an equivalence $\XMod(\A)\simeq\Grpd(\A)$.

\begin{example}
\label{ex:crossed point associated to trivial crossed module}
Consider the pre-crossed module $(X\xrightarrow{0}A,\tau^A_X)$ given by the trivial action. Then the situation simplifies, and the associated reflexive graph is
\begin{align*}
\xymatrix{
X \ar@{{ |>}->}[rr]^-{\langle 0,1_X\rangle} && A\times X \ar@<1ex>[rr]^-{\pi_A} \ar@<-1ex>[rr]_-{\pi_A} && A.\ar[ll]|-{\langle 1_A,0\rangle}
}
\end{align*}
Furthermore, from Lemma~\ref{lemma: mult struc <=> HUQ=0} it follows that $(X\xrightarrow{0}A,\tau^A_X)$ is a crossed module if and only if $X$ is an abelian object.
\end{example}

\section{Double groupoids and double reflexive graphs}
\label{Section:Double groupoids and double reflexive graphs}

We recall the categories of internal double groupoids and internal double reflexive graphs, and describe how one is embedded into the other as a reflective subcategory. In this section, $\A$ is a semi-abelian category that satisfies \SH.

\begin{definition}
A \emph{double reflexive graph in $\A$} is a reflexive graph in $\RG(\A)$. This means that the category $\RG^2(\A)$ is defined as $\RG(\RG(\A))$.
\end{definition}

\begin{lemma}
A double reflexive graph can be depicted as a diagram in $\A$ of the form in Figure~\ref{Star} in which each pair of adjacent vertices forms a reflexive graph.
\end{lemma}

\begin{definition}
An \emph{internal double groupoid in $\A$} is a groupoid in $\Grpd(\A)$. This means that the category $\Grpd^2(\A)$ is defined as $\Grpd(\Grpd(\A))$.
\end{definition}

\begin{proposition}
Double groupoids are diagrams as in Figure~\ref{Star} in which each reflexive graph has an internal groupoid structure.
\end{proposition}
\begin{proof}
This combines two facts: internal groupoid structures are necessarily unique, and limits in functor categories are computed pointwise.
\end{proof}

\subsection{Double groupoids induced by particular double reflexive graphs.}
\label{subsection:double groupoids induced by particular double reflexive graphs}

Consider a double reflexive graph as in Figure~\ref{Fig:BigDiagram}, such that the reflexive graph on the right and the one at the bottom are already groupoids; in other words (Lemma~\ref{lemma: mult struc <=> HUQ=0}), ${[K_{d_R},K_{c_R}]}$ and ${[K_{d_D},K_{c_D}]}$ are trivial. We want to construct a double groupoid by dividing the join
\[
{[K_{d_L},K_{c_L}]}^{\Qq}_{A}\vee {[K_{d_U},K_{c_U}]}^{\Qq}_{A}
\]
out of $A$. This does indeed work, and is part of the construction in the following:

\begin{proposition}\label{Proposition Induced Double Groupoid}
The forgetful inclusion $\Grpd^2(\A)\to \RG^2(\A)$ has a left adjoint $\RG^2(\A)\to \Grpd^2(\A)$.
\end{proposition}

Let us explain in detail how this works. Recall that in a semi-abelian category, the join of two subobjects $A$, $B$ of an object $X$ may be obtained via the image factorisation 
\[
\xymatrix{A+B \ar@{-{ >>}}[r] & A\vee B \ar@{{ >}->}[r]^-{a\vee b} & X}
\]
of the map $\binom{a}{b}\colon A+B\to X$ induced by representing monomorphisms $a$, $b$. When both $a$ and $b$ are normal monomorphisms, then so is $a\vee b$. In particular, we have:

\begin{lemma}[\cite{BB04}]
\label{lemma:the kernel of the diagonal of a pushout of two regular epimorphisms is the join of the two kernels}
Given two regular epimorphisms and their pushout
\[
\xymatrix{
A \ar@{-{ >>}}[r]^-{f} \ar@{-{ >>}}[d]_-{g} \ar@{.{ >>}}[rd]|-{h} & B \ar@{-{ >>}}[d]^-{f_*(g)}\\
C \ar@{-{ >>}}[r]_-{g_*(f)} & D
}
\]
the kernel of the diagonal $h$ is the join of the kernels of $f$ and $g$:  $K_{h}= K_{f}\vee K_{g}$.
\end{lemma}

We apply this to the situation depicted in Figure~\ref{Fig:BigDiagram}, where $
S\coloneq[K_{d_L},K_{c_L}]$ and $T\coloneq[K_{d_U},K_{c_U}]$. All dotted arrows are universally induced: to see this, use Lemma~\ref{lemma: mult struc <=> HUQ=0} and the fact that the reflexive graph on the right and the one at the bottom are groupoids. 
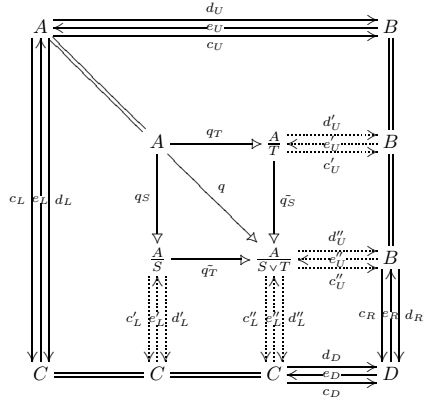
\begin{figure}
\resizebox{.45\textwidth}{!}{
$\vcenter{\xymatrix@!0@=6em{
A \ar@<1ex>[ddd]^-{d_L} \ar@<-1ex>[ddd]_-{c_L} \ar@<1ex>[rrr]^-{d_U} \ar@<-1ex>[rrr]_-{c_U} \ar@{=}[rd] &&& B \ar[lll]|-{e_U} \ar@{=}[d]\\
& A \ar@{-{ >>}}[rd]^-{q} \ar@{-{ >>}}[r]^-{q_T} \ar@{-{ >>}}[d]_-{q_S} & \frac{A}{T} \ar@{.>}@<1ex>[r]^-{d'_U} \ar@{.>}@<-1ex>[r]_-{c'_U} \ar@{-{ >>}}[d]^-{\Tilde{q_S}} & B \ar@{=}[d] \ar@{.>}[l]|-{e'_U}\\
& \frac{A}{S} \ar@{-{ >>}}[r]_-{\Tilde{q_T}} \ar@{.>}@<1ex>[d]^-{d'_L} \ar@<-1ex>@{.>}[d]_-{c'_L} & \frac{A}{S\vee T} \ar@{.>}@<1ex>[d]^-{d''_L} \ar@{.>}@<-1ex>[d]_-{c''_L} \ar@{.>}@<1ex>[r]^-{d''_U} \ar@{.>}@<-1ex>[r]_-{c''_U} & B \ar@{.>}[l]|-{e''_U} \ar@<1ex>[d]^-{d_R} \ar@<-1ex>[d]_-{c_R}\\
C \ar[uuu]|-{e_L} \ar@{=}[r] & C \ar@{.>}[u]|-{e'_L} \ar@{=}[r] & C \ar@{.>}[u]|-{e''_L} \ar@<1ex>[r]^-{d_D} \ar@<-1ex>[r]_-{c_D} & D \ar[l]|-{e_D} \ar[u]|-{e_R}
}}
$}
\caption{Reflecting a particular double reflexive graph to a double groupoid}\label{Fig:BigDiagram} 
\end{figure}
Clearly, 
\[
\text{$({A}/{(S\vee T)}, B, d''_U,c''_U,e''_U)$\qquad and\qquad $({A}/{(S\vee T)}, C, d''_L,c''_L,e''_L)$}
\]
are reflexive graphs. We need to prove that they are internal groupoids, that is
\[
\label{eq: the induced graphs are multiplicative}
\text{${[K_{d_U''},K_{c_U''}]}^{\Qq}_{\frac{A}{S\vee T}}=0$\qquad and\qquad ${[K_{d_L''},K_{c_L''}]}^{\Qq}_{\frac{A}{S\vee T}}=0$.}
\]
We shall only prove the first equality, since the strategy for the second one is the same. 
Consider the following diagram, which has vertical and horizontal short exact sequences by the $3\times 3$-Lemma.
\[
\xymatrix{
K_{\Tilde{q_S}} \ar@{{ |>}->}[r] \ar@{=}[d] & K_{d_U'} \ar@{-{ >>}}[r] \ar@{{ |>}->}[d] & K_{d_U''} \ar@{{ |>}->}[d]\\
K_{\Tilde{q_S}} \ar@{{ |>}->}[r] \ar@{-{ >>}}[d] & \frac{A}{T} \ar@{-{ >>}}[r]^-{\Tilde{q_S}} \ar@{-{ >>}}[d]_-{d'_U} & \frac{A}{S\vee T} \ar@{-{ >>}}[d]^-{d''_U}\\
0 \ar@{{ |>}->}[r] & B \ar@{=}[r] & B
}
\]
In precisely the same way it is possible to describe the image factorisation of $\Tilde{q_S}\comp k_{c_U'}$. Now we can apply Proposition~\ref{Higgins properties} to the diagram 
\[
\xymatrix{
K_{d_U'} \ar@{-{ >>}}[d] \ar@{{ |>}->}[r] & \frac{A}{T} \ar@{-{ >>}}[d]^-{\Tilde{q_S}} & K_{c_U'} \ar@{-{ >>}}[d] \ar@{{ |>}->}[l]\\
K_{d_U''} \ar@{{ |>}->}[r] & \frac{A}{S\vee T} & K_{c_U''} \ar@{{ |>}->}[l]
}
\]
to see that if ${[K_{d_U'},K_{c_U'}]}$ is trivial, then ${[K_{d_U''},K_{c_U''}]}$ is trivial as well.

\begin{proposition}
\label{prop:the previous construction is an instance of the reflector}
Given a double reflexive graph as in Figure~\ref{Fig:BigDiagram} where the reflexive graphs on the bottom and on the right are internal groupoids, the morphism of double reflexive graphs in Figure~\ref{Fig:Unit} 
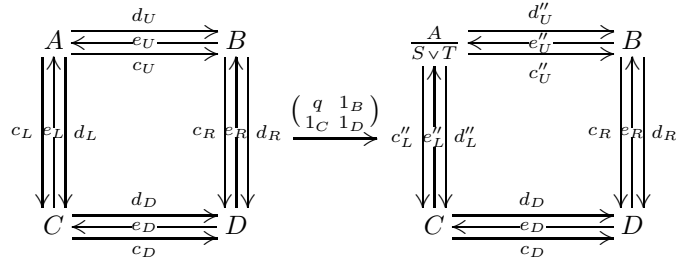
\begin{figure}
$\vcenter{
\xymatrix{
A \ar@<1ex>[dd]^-{d_L} \ar@<-1ex>[dd]_-{c_L} \ar@<1ex>[rr]^-{d_U} \ar@<-1ex>[rr]_-{c_U} && B \ar[ll]|-{e_U} \ar@<1ex>[dd]^-{d_R} \ar@<-1ex>[dd]_-{c_R} && \frac{A}{S\vee T} \ar@<1ex>[dd]^-{d_L''} \ar@<-1ex>[dd]_-{c_L''} \ar@<1ex>[rr]^-{d_U''} \ar@<-1ex>[rr]_-{c_U''} && B \ar[ll]|-{e_U''} \ar@<1ex>[dd]^-{d_R} \ar@<-1ex>[dd]_-{c_R}\\
&& \ar@{}[rr]^(.25){}="a"^(.75){}="b" \ar "a";"b"^-{\morph{q & 1_B}{1_C& 1_D}} && \\
C \ar[uu]|-{e_L} \ar@<1ex>[rr]^-{d_D} \ar@<-1ex>[rr]_-{c_D} & & D \ar[uu]|-{e_R} \ar[ll]|-{e_D} && C \ar[uu]|-{e_L''} \ar@<1ex>[rr]^-{d_D} \ar@<-1ex>[rr]_-{c_D} && D \ar[uu]|-{e_R} \ar[ll]|-{e_D}
}
}
$\caption{The unit of the adjunction, when we already have a groupoid on the right and on the bottom}\label{Fig:Unit}
\end{figure}
coincides with the unit of the adjunction between double groupoids and double reflexive graphs induced by the adjunction between groupoids and reflexive graphs.
\end{proposition}
\begin{proof}
Consider another morphism of double reflexive graphs as in Figure~\ref{Fig:Morphism}
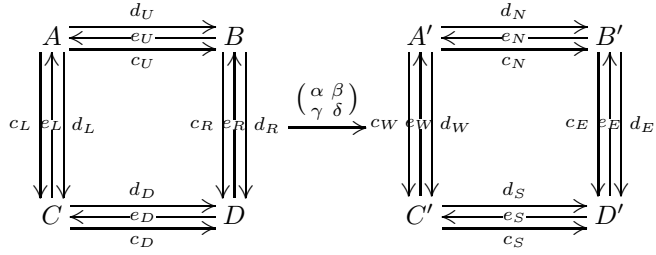
\begin{figure}
$\vcenter{
\xymatrix{
A \ar@<1ex>[dd]^-{d_L} \ar@<-1ex>[dd]_-{c_L} \ar@<1ex>[rr]^-{d_U} \ar@<-1ex>[rr]_-{c_U} && B \ar[ll]|-{e_U} \ar@<1ex>[dd]^-{d_R} \ar@<-1ex>[dd]_-{c_R} && A' \ar@<1ex>[dd]^-{d_W} \ar@<-1ex>[dd]_-{c_W} \ar@<1ex>[rr]^-{d_N} \ar@<-1ex>[rr]_-{c_N} && B' \ar[ll]|-{e_N} \ar@<1ex>[dd]^-{d_E} \ar@<-1ex>[dd]_-{c_E}\\
&& \ar@{}[rr]^(.25){}="a"^(.75){}="b" \ar "a";"b"^-{\morph{\alpha & \beta}{\gamma & \delta}} && \\
C \ar[uu]|-{e_L} \ar@<1ex>[rr]^-{d_D} \ar@<-1ex>[rr]_-{c_D} & & D \ar[uu]|-{e_R} \ar[ll]|-{e_D} && C' \ar[uu]|-{e_W} \ar@<1ex>[rr]^-{d_S} \ar@<-1ex>[rr]_-{c_S} && D' \ar[uu]|-{e_E} \ar[ll]|-{e_S}
}
}
$\caption{A morphism of double reflexive graphs}\label{Fig:Morphism}
\end{figure}
in which the codomain is a double groupoid. We want to define a morphism $\phi\colon{A}/{(S\vee T)}\to A'$ such that $\phi\comp q=\alpha$. In order to do this, consider the diagrams
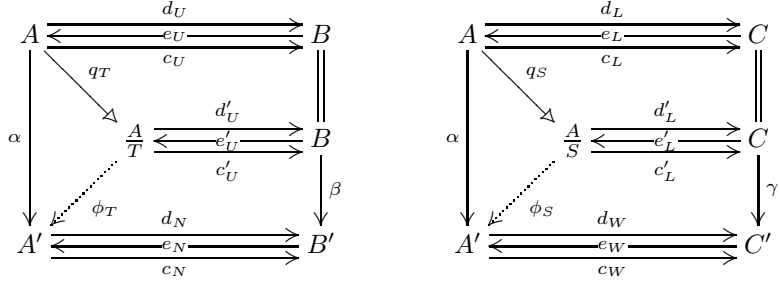
\begin{align*}
\xymatrix{
A \ar[dd]_-{\alpha} \ar@{-{ >>}}[rd]^-{q_T} \ar@<1ex>[rrr]^-{d_U} \ar@<-1ex>[rrr]_-{c_U} &&& B \ar@{=}[d] \ar[lll]|-{e_U}\\
& \frac{A}{T} \ar@{.>}[dl]^-{\phi_T} \ar@<1ex>[rr]^-{d_U'} \ar@<-1ex>[rr]_-{c_U'} && B \ar[d]^-{\beta} \ar[ll]|-{e_U'}\\
A' \ar@<1ex>[rrr]^-{d_N} \ar@<-1ex>[rrr]_-{c_N} &&& B' \ar[lll]|-{e_N}
}
&&
\xymatrix{
A \ar[dd]_-{\alpha} \ar@{-{ >>}}[rd]^-{q_S} \ar@<1ex>[rrr]^-{d_L} \ar@<-1ex>[rrr]_-{c_L} &&& C \ar@{=}[d] \ar[lll]|-{e_L}\\
& \frac{A}{S} \ar@{.>}[dl]^-{\phi_S} \ar@<1ex>[rr]^-{d_L'} \ar@<-1ex>[rr]_-{c_L'} && C \ar[d]^-{\gamma} \ar[ll]|-{e_L'}\\
A' \ar@<1ex>[rrr]^-{d_W} \ar@<-1ex>[rrr]_-{c_W} &&& C' \ar[lll]|-{e_W}
}
\end{align*}
using the notation of in Figure~\ref{Fig:BigDiagram}.
Here the dotted maps are defined through the universal property of the unit of the adjunction between $\RG(\A)$ and $\Grpd(\A)$. Now we define $\phi$ via the universal property of the pushout in the diagram 
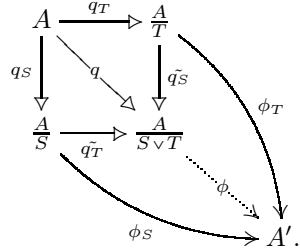
\[
\xymatrix{
A \ar@{-{ >>}}[rd]|-{q} \ar@{-{ >>}}[r]^-{q_T} \ar@{-{ >>}}[d]_-{q_S} & \frac{A}{T} \ar@{-{ >>}}[d]^-{\Tilde{q_S}} \ar@/^1pc/[rdd]^-{\phi_T}\\
\frac{A}{S} \ar@{-{ >>}}[r]_-{\Tilde{q_T}} \ar@/_1pc/[rrd]_-{\phi_S} & \frac{A}{S\vee T} \ar@{.>}[rd]|-{\phi}\\
&& A'.
}
\]
Obviously, $\morph{\phi & \beta}{\gamma & \delta}$ is a morphism of double groupoids, and it is the only one where $\morph{\phi & \beta}{\gamma & \delta}\comp\morph{q & 1_B}{1_C & 1_D}=\morph{\alpha & \beta}{\gamma & \delta}$.
\end{proof}

Proposition~\ref{Proposition Induced Double Groupoid} follows as a straightforward consequence of this result; the main thing to be done is to reflect the bottom and the right reflexive graph in Figure~\ref{Fig:BigDiagram} to internal groupoids before applying Proposition~\ref{prop:the previous construction is an instance of the reflector}.

\section{Crossed squares of groups and internal crossed squares}
\label{Section:crossed squares}
Crossed squares are to double groupoids what crossed modules are to groupoids. Before studying this principle in the semi-abelian context, we recall the definition in the category of groups. The case of Lie algebras shall be treated much later, in Subsection~\ref{SubsecLie}.

\begin{definition}[\cite{GWL80,Lod82,BL87}]
\label{defi:explicit defi of crossed square of groups}
A \emph{crossed square (of groups)} is a commutative square
\[
\xymatrix{
P \ar[r]^-{p_M} \ar[d]_-{p_N} & M \ar[d]^-{\mu}\\
N \ar[r]_-{\nu} & L
}
\]
in $\Grp$, together with actions of $L$ on $M$, $N$ and $P$ (and hence actions of $M$ on $P$ and $N$ via $\mu$, and of $N$ on $M$ and $P$ via $\nu$) and a function (not a group morphism!) $h\colon M\times N\to P$ such that the following axioms hold:
\begin{enumerate}[{\rm (X.1)}]
\setcounter{enumi}{-1}
\item $h(mm',n)={}^mh(m',n)h(m,n)$ and $h(m,nn')=h(m,n){}^nh(m,n')$;
\item the maps $p_M$ and $p_N$ preserve the actions of $L$, furthermore with the given actions $(M\xrightarrow{\mu}L)$, $(N\xrightarrow{\nu}L)$ and $(P\xrightarrow{\mu\circ p_M=\nu\circ p_N}L)$ are crossed modules;
\item $p_M(h(m,n))=m{}^{n}m^{-1}$ and $p_N(h(m,n))={}^{m}nn^{-1}$;
\item $h(p_M(p),n)=p{}^{n}p^{-1}$ and $h(m,p_N(p))={}^{m}pp^{-1}$;
\item ${}^lh(m,n)=h({}^lm,{}^ln)$;
\end{enumerate}
for all $l\in L$, $m$, $m'\in M$, $n$, $n'\in N$ and $p\in P$.

A map of crossed squares is given by four group morphisms which are compatible with the actions and with the map $h$. Crossed squares and morphisms between them form the category $\XSqr(\Grp)$.
\end{definition}

\begin{definition}[\cite{BL87}]
\label{defi:crossed pairing}
Given a pair of $L$-crossed modules $(M\xrightarrow{\mu}L,\xi_M)$ and $(N\xrightarrow{\nu}L,\xi_N)$ in $\Grp$, we have an action $\xi^M_N$ of $M$ on $N$ induced via $\mu$ and an action $\xi^N_M$ of $N$ on $M$ induced via $\nu$. We say that a map $h\colon M\times N \to P$ is a \emph{crossed pairing} if the following hold, for each $m$, $m'\in M$ and $n$, $n'\in N$:
\[
\text{$h(mm',n)=h({}^mm',{}^mn)h(m,n)$,\qquad\qquad $h(m,nn')=h(m,n)h({}^nm,{}^nn')$.}
\]
\end{definition}

\begin{remark}
\label{rmk:h of a crossed square is a crossed pairing}
Notice that if we have a crossed square, then the map $h\colon M\times N\to P$ is actually a crossed pairing. Indeed, by using \X4 and the fact that the actions involved are induced from the actions of $L$, we can show the equivalence between condition \X0 and $h$ being a crossed pairing, through the equalities
\begin{align*}
{}^mh(m',n)&={}^{\mu(m)}h(m',n)=h({}^{\mu(m)}m',{}^{\mu(m)}n)=h({}^mm',{}^mn),\\
{}^nh(m,n')&={}^{\nu(n)}h(m,n')=h({}^{\nu(n)}m,{}^{\nu(n)}n')=h({}^nm,{}^nn').
\end{align*}
\end{remark}

In Proposition~5.2 in~\cite{Lod82}, it is proved that Definition~\ref{defi:explicit defi of crossed square of groups} is equivalent to the concept of a \emph{$cat^2$-group}. Using that crossed modules are equivalent to $cat^1$-groups, that is using the equivalences of categories
\[
\text{$\cat^1$-$\Grp$}\simeq\Grpd(\Grp)\simeq\XMod(\Grp),
\]
we may view a crossed square as an internal crossed module in the category of crossed modules of groups. This means that we have equivalences
\[
\XSqr(\Grp)\simeq\text{$\cat^2$-$\Grp$}\simeq\Grpd^2(\Grp)\simeq\XMod(\XMod(\Grp)).
\]
In particular, the functor from $\Grpd^2(\Grp)$ to $\XSqr(\Grp)$ is given by normalisation; that is, given a double groupoid as in Figure~\ref{Figure:Normalisation} 
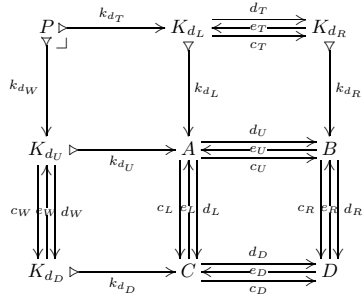
\begin{figure}

\resizebox{.4\textwidth}{!}
{
\xymatrixcolsep{4pc}
\xymatrixrowsep{4pc}
\xymatrix{
P \pullbackcorner \ar@{{ |>}->}[r]^-{k_{d_T}} \ar@{{ |>}->}[d]_-{k_{d_W}} & K_{d_L} \ar@<1ex>[r]^-{d_T} \ar@<-1ex>[r]_-{c_T} \ar@{{ |>}->}[d]^-{k_{d_L}} & K_{d_R} \ar[l]|-{e_T} \ar@{{ |>}->}[d]^-{k_{d_R}}\\
K_{d_U} \ar@<1ex>[d]^-{d_W} \ar@<-1ex>[d]_-{c_W} \ar@{{ |>}->}[r]_-{k_{d_U}} & A \ar@<1ex>[d]^-{d_L} \ar@<-1ex>[d]_-{c_L} \ar@<1ex>[r]^-{d_U} \ar@<-1ex>[r]_-{c_U} & B \ar[l]|-{e_U} \ar@<1ex>[d]^-{d_R} \ar@<-1ex>[d]_-{c_R}\\
K_{d_D} \ar[u]|-{e_W} \ar@{{ |>}->}[r]_-{k_{d_D}} & C \ar[u]|-{e_L} \ar@<1ex>[r]^-{d_D} \ar@<-1ex>[r]_-{c_D} & D \ar[u]|-{e_R} \ar[l]|-{e_D}
}
}

\caption{The normalisation of a double groupoid}\label{Figure:Normalisation}
\end{figure}
where the outer square is obtained by taking kernels of the domain morphisms, the induced maps admit suitable internal actions induced by the conjugation in $A$, making it a crossed square. Similarly, a morphism of internal crossed squares is the (unique) normalisation of a morphism of double groupoids.

We return to the context of semi-abelian categories.

\begin{definition}
\label{defi:implicit defi of crossed square}
In a semi-abelian category $\A$ that satisfies \SH, an \emph{internal crossed square} is an internal crossed module in $\XMod(\A)$. This means that the category $\XSqr(\A)$ is defined as $\XMod(\XMod(\A))$.
\end{definition}

We would like to have an explicit description of an internal crossed square \emph{as a diagram in the underlying category $\A$} like in the case of groups, but this is far from straightforward. Certainly any double groupoid can be normalised to a commutative square as in Figure~\ref{Figure:Normalisation}, and it is also possible to deduce suitable actions. The normalisation is the \lq\lq underlying commutative square\rq\rq\ of the given crossed module of crossed modules, so we have a forgetful functor. This raises the question, what kind of structure needs to be added to the square so that this forgetful functor can be lifted to an equivalence. In other words, we are confronted with a kind of \emph{descent problem}. Part of the aim of the paper is to answer this question, and we actually manage to provide partial answers in several special cases: the concept of a weak crossed square---see Section~\ref{section:towards crossed squares through the non-abelian tensor product}---does it for groups, Lie algebras, and the case where we find a pairing that happens to be a suitable regular epimorphism. 

For now, let us consider a basic example and prove some preliminary results. In Section~\ref{Section Non-Abelian Tensor} we use the idea of a crossed square in the definition of the non-abelian tensor product.

\begin{example}\label{Example intersection}
Given two normal subobjects $M$, $N\lhd L$ of an object $L$ in a semi-abelian category with \SH, the square induced by taking their intersection (the pullback of their representing monomorphisms) carries a canonical crossed square structure. Indeed, by taking cokernels, any pullback square of normal monomorphisms is seen to be part of a $3\times 3$-diagram; replacing the kernels by kernel pairs we find a \lq\lq denormalised $3\times 3$-diagram\rq\rq\ as in~\cite{Bou03}. If~$(R,r_1,r_2)$ and $(S,s_1,s_2)$ denote the respective denormalisations of $M$ and $N$ (the kernel pairs of the cokernels of their inclusions into $L$), then the upper left corner of this diagram is a double equivalence relation as in Figure~\ref{Figure Square} on the left, which may be constructed as the pullback on the right. 
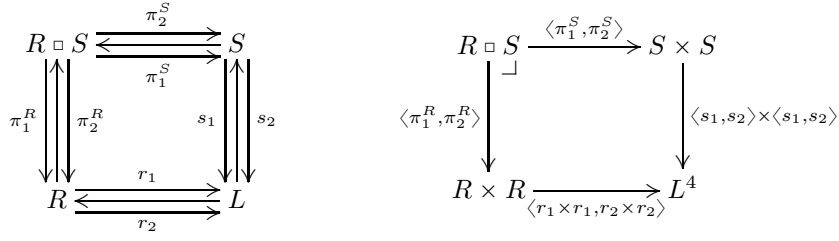
\begin{figure}
\[
\vcenter{\xymatrix@=2em{R\square S \ar@<1ex>[rr]^-{\pi_2^S} \ar@<-1ex>[rr]_-{\pi_1^S}\ar@<1ex>[dd]^-{\pi_2^R}\ar@<-1ex>[dd]_-{\pi_1^R}&&S\ar[ll]\ar@<1ex>[dd]^-{s_2}\ar@<-1ex>[dd]_-{s_1}\\\\
R \ar[uu]\ar@<1ex>[rr]^-{r_1}\ar@<-1ex>[rr]_-{r_2}&& L\ar[ll]\ar[uu]}}
\qquad\qquad
\vcenter{\xymatrix@=4em{R\square
S \pullbackcorner \ar[r]^-{\langle\pi_{1}^{S},\pi_{2}^{S}\rangle} \ar[d]_-{\langle\pi_{1}^{R},\pi_{2}^{R}\rangle} & S \times S\ar[d]^-{\langle s_1,s_2\rangle\times\langle s_1,s_2\rangle}\\
R\times R\ar[r]_-{\langle r_1\times r_1
,r_2\times r _2\rangle}& L^4 }}
\]
\caption{The parallelistic double equivalence relation $R\square S$ and its construction}\label{Figure Square}
\end{figure}
Hence it is a double groupoid, which forgets by normalisation to the given intersection of normal subobjects, viewed as crossed modules. 
\end{example}

\subsection{The diagonal internal crossed module in an internal crossed square.}\label{rmk:definitions of diagonal action}

We find ourselves in a semi-abelian category that satisfies \SH. Referring to Figure~\ref{Figure:Normalisation} we will write $j$ for the diagonal of the upper left square, with $(D,A,c,d,e)$ the reflexive graph structure induced diagonally in the lower right square ($c=c_D\comp c_L$, $d=d_D\comp d_L$, $e=e_L\comp e_D$) and $\lambda=c\comp j$.

Given an internal double groupoid as in Figure~\ref{Figure:Normalisation} we can define an action of $D$ on $P$ in the following different ways:
\begin{itemize}
\item First of all we can define it as the dotted arrow in the diagram
\begin{equation}
\label{diag:first definition of diagonal action}
\vcenter{
\xymatrixcolsep{5pc}
\xymatrix{
D\flat P \ar@{.>}[d]_-{\xi} \ar[r]^-{e\flat k} & A\flat A \ar[d]^-{\chi_A} \ar[r]^-{\langle d_U,d_L\rangle\flat \langle d_U,d_L\rangle} & (B\times C)\flat(B\times C) \ar[d]^-{\chi_{(B\times C)}}\\
P \ar@{{ |>}->}[r]_-{j} & A \ar[r]_-{\langle d_U,d_L\rangle} & B\times C
}
}
\end{equation}
where $j=k_{d_L}\comp k_{d_T}=k_{d_U}\comp k_{d_W}$ is the kernel of $\langle d_U,d_L\rangle$;
\item next we induce it through any of the diagrams
\[
\vcenter{
\xymatrixcolsep{3pc}
\xymatrix{
D\flat P \ar@{.>}[d]_-{\xi} \ar[r]^-{e_R\flat k_{d_W}} & B\flat K_{d_U} \ar[d]^-{\psi_U} \ar[r]^-{d_R\flat d_W} & D\flat K_{d_D} \ar[d]^-{\psi_D}\\
P \ar@{{ |>}->}[r]_-{k_{d_W}} & K_{d_U} \ar[r]_-{d_W} & K_{d_D}
}
}
\quad
\vcenter{
\xymatrixcolsep{3pc}
\xymatrix{
D\flat P \ar@{.>}[d]_-{\xi} \ar[r]^-{e_D\flat k_{d_T}} & C\flat K_{d_L} \ar[d]^-{\psi_L} \ar[r]^-{d_D\flat d_T} & D\flat K_{d_R} \ar[d]^-{\psi_R}\\
P \ar@{{ |>}->}[r]_-{k_{d_T}} & K_{d_L} \ar[r]_-{d_T} & K_{d_R}.
}
}
\]
\end{itemize}
Notice that these three actions are uniquely determined by the universal property of the kernels and that they are actually the same: indeed it suffices to show that if such a $\xi$ makes one of the previous diagrams commute, then also the other does. This is easily seen via the diagrams
\begin{align*}
\xymatrixcolsep{3pc}
\vcenter{\xymatrix{
D\flat P \ar@{.>}[d]_-{\xi} \ar[r]^-{e_R\flat k_{d_W}} & B\flat K_{d_U} \ar[d]^-{\psi_U} \ar[r]^-{e_U\flat k_{d_U}} & A\flat A \ar[d]^-{\chi_A}\\
P \ar@{{ |>}->}[r]_-{k_{d_W}} & K_{d_U} \ar@{{ |>}->}[r]_-{k_{d_U}} & A
}}
\;{\text{and}}\;
\xymatrixcolsep{3pc}
\vcenter{\xymatrix{
D\flat P \ar@{.>}[d]_-{\xi} \ar[r]^-{e_D\flat k_{d_T}} & C\flat K_{d_L} \ar[d]^-{\psi_L} \ar[r]^-{e_D\flat k_{d_L}} & A\flat A \ar[d]^-{\chi_A}\\
P \ar@{{ |>}->}[r]_-{k_{d_T}} & K_{d_L} \ar@{{ |>}->}[r]_-{k_{d_L}} & A.
}}
\end{align*}
In each rectangle the square on the right commutes by Remark~\ref{rmk:acting is like conjugating in the semidirect product}. Therefore, since $k_{d_U}$ and $k_{d_L}$ are monomorphisms, the square on the left in each rectangle commutes if and only if the corresponding outer rectangle does. These, however, coincide with the left hand square in~\eqref{diag:first definition of diagonal action}. Hence the three definitions are the same.

We can also define an action of $D$ on $P$ in the following, \emph{a priori} different, way. Consider the diagram
\[
\xymatrix{
P \ar@{{ |>}->}[r]^-{k_{d_T}} \ar@{.>}[rd]|l \ar@{{ |>}->}[d]_-{k_{d_W}} & K_{d_L} \ar[d] \ar@{=}[r] & K_{d_L} \ar@{{ |>}->}[dd]^-{k_{d_L}}\\
K_{d_U} \ar[r] \ar@{=}[d] & K_{d_L}\vee K_{d_U} \ar@{{ |>}->}[rd]|-{k_{d_L}\vee k_{d_U}}\\
K_{d_U} \ar@{{ |>}->}[rr]_-{k_{d_U}} && A.
}
\]
Consider the diagonal split extension
\[
\xymatrixcolsep{3pc}
\xymatrix{
K_{d_L}\vee K_{d_U} \ar@{{ |>}->}[r]^-{k_d} & A \ar@<.5ex>[r]^-{d} & D \ar@<.5ex>[l]^-{e}
}
\]
defined through Figure~\ref{Figure:Normalisation}. Notice that $K_{d_L}\vee K_{d_U}$ is the kernel of $d$ (and $k_d={k_{d_L}\vee k_{d_U}}$) because of Lemma~\ref{lemma:the kernel of the diagonal of a pushout of two regular epimorphisms is the join of the two kernels} and Lemma~\ref{lemma:epimorphism of points is pushout}. Since $k_d\comp l$ is a normal monomorphism, we can construct the diagram
\begin{equation}
\label{diag:definition of diagonal groupoid}
\vcenter{
\xymatrixcolsep{3pc}
\xymatrix{
P \ar@{{ |>}->}[d]_-{l} \ar@{{ |>}.>}[r]^-{k_{\Hat{d}}} & \Hat{A} \ar@{.>}[d]_-{l\rtimes 1_D} \ar@<.5ex>@{.>}[r]^-{\Hat{d}} & D \ar@{=}[d] \ar@<.5ex>@{.>}[l]^-{\Hat{e}}\\
K_{d_L}\vee K_{d_U} \ar@{{ |>}->}[r]_-{k_d} & A \ar@<.5ex>[r]^-{d} & D \ar@<.5ex>[l]^-{e}\\
}
}
\end{equation}
through Lemma~2.6 in~\cite{CGVdL15a}, which gives us an action of $D$ on $P$.

\begin{lemma}
\label{lemma:those actions are the same}
The two internal actions defined above coincide.
\end{lemma}
\begin{proof}
In order to show this, it suffices to prove that the equivalence $\Act(\A)\simeq \Pt(\A)$ sends the point constructed in~\eqref{diag:definition of diagonal groupoid} into the action $\xi$ uniquely defined through the commutativity of~\eqref{diag:first definition of diagonal action}. To do this, consider the diagram
\[
\xymatrixcolsep{3pc}
\xymatrix{
D\flat P \ar@{.>}[d] \ar@{{ |>}->}[r]^-{k_{D,P}} & D+P \ar[d]_-{\binom{\Hat{e}}{k_{\Hat{d}}}} \ar@<.5ex>[r]^-{\binom{1_D}{0}} & D \ar@{=}[d] \ar@<.5ex>[l]^-{i_D}\\
P \ar@{{ |>}->}[r]_-{k_{\Hat{d}}} & \Hat{A} \ar@<.5ex>[r]^-{\Hat{d}} & D. \ar@<.5ex>[l]^-{\Hat{e}}
}
\]
Let us prove that $k_{\Hat{d}}\comp\xi=\binom{\Hat{e}}{k_{\Hat{d}}}\comp k_{D,P}$. The map $l\rtimes 1_D$ is a monomorphism since $l$ is so, therefore we need to show that $(l\rtimes 1_D)\comp k_{\Hat{d}}\comp\xi=(l\rtimes 1_D)\comp\binom{\Hat{e}}{k_{\Hat{d}}}\comp k_{D,P}$. The left hand side is equal to $k\comp \xi$ which in turn (by definition of $\xi$) is $\chi_A\comp (e\flat k)$, whereas the chain of equalities
\begin{align*}
(l\rtimes 1_D)\comp\binom{\Hat{e}}{k_{\Hat{d}}}\comp k_{D,P}&=\binom{(l\rtimes 1_D)\circ\Hat{e}}{(l\rtimes 1_D)\circ k_{\Hat{d}}}\comp k_{D,P}
=\binom{e}{k}\comp k_{D,P}=\binom{1_A}{1_A}\comp(e+k)\comp k_{D,P}\\
&=\binom{1_A}{1_A}\comp k_{A,A}\comp (e\flat k)=\chi_A\comp (e\flat k)
\end{align*}
gives us the right hand side.
\end{proof}

\begin{proposition}
\label{prop:the diagonal is a crossed module}
In the situation above, $(P\xrightarrow{\lambda}D,\xi)$ is an internal crossed module. 
\end{proposition}
\begin{proof}
Notice that if we define $\Hat{c}\coloneq c\comp (l\rtimes 1_D)$, then we have that $\Hat{c}\comp k_{\Hat{d}}=c\comp j=\lambda$. Therefore it suffices to show that the first row in~\eqref{diag:definition of diagonal groupoid} is actually a groupoid, once it is endowed with $\Hat{c}$ as a second leg. In order to prove this, since $P=K_{\Hat{d}}$, by Lemma~\ref{lemma: mult struc <=> HUQ=0} we only need to show that $[P,K_{\Hat{c}}]$ is trivial. But $K_{\Hat{c}}\leq K_c$ implies $[P,K_{\Hat{c}}]\leq [P,K_c]$, hence it suffices that $[P,K_c]=0$. This follows from the chain of inequalities of subobjects
\begin{align*}
[P,K_c]&= [P,K_{c_U}\vee K_{c_L}]
= [P,K_{c_U}]\vee [P,K_{c_L}]\vee [P,K_{c_U},K_{c_L}]\\
&\leq [K_{d_U},K_{c_U}]\vee [K_{d_L},K_{c_L}]\vee [K_{c_U},K_{c_U},A]
= 0\vee 0\vee 0= 0
\end{align*}
which we have by Lemma~\ref{lemma:the kernel of the diagonal of a pushout of two regular epimorphisms is the join of the two kernels} and Proposition~\ref{Higgins properties}.
\end{proof}

\begin{proposition}
\label{prop:every morphism of internal crossed squares has an underlying morphism of internal crossed modules between the diagonals}
Given a morphism of internal double groupoids
\[
\xymatrix@!0@=2em{
A \ar@<1ex>[dd] \ar@<-1ex>[dd] \ar@<1ex>[rr] \ar@<-1ex>[rr] && B \ar[ll] \ar@<1ex>[dd] \ar@<-1ex>[dd] &&&& A' \ar@<1ex>[dd] \ar@<-1ex>[dd] \ar@<1ex>[rr] \ar@<-1ex>[rr] && B' \ar[ll] \ar@<1ex>[dd] \ar@<-1ex>[dd]\\
&& \ar@{}[rrrr]^(.25){}="a"^(.75){}="b" \ar "a";"b"^-{\morph{\alpha & \beta}{\gamma & \delta}} &&&& \\
C \ar[uu] \ar@<1ex>[rr] \ar@<-1ex>[rr] & & D \ar[uu] \ar[ll] &&&& C' \ar[uu] \ar@<1ex>[rr] \ar@<-1ex>[rr] && D' \ar[uu] \ar[ll]
}
\]
consider the unique morphism of internal crossed squares induced between their normalisations, and denote $\rho\colon P\to P'$ the upper-left component. Then 
\[
(P\xrightarrow{\lambda}D,\xi)\xlongrightarrow{(\rho,\delta)}(P'\xrightarrow{\lambda'}D',\xi')
\]
is a morphism of internal crossed modules.
\end{proposition}
\begin{proof}
We want to show the commutativity of the diagrams
\begin{align*}
\vcenter{\xymatrix{
P \ar[d]_-{\rho} \ar[r]^-{\lambda} & D \ar[d]^-{\delta}\\
P' \ar[r]_-{\lambda'} & D'
}}
\qquad\text{and}\qquad
\vcenter{\xymatrix{
D\flat P \ar[d]_-{\delta\flat\rho} \ar[r]^-{\xi} & P \ar[d]^-{\rho}\\
D'\flat P' \ar[r]_-{\xi'} & P'.
}}
\end{align*}
The first one is obvious by construction of the map $\rho$. For the second one we need to use one of the explicit constructions for the actions $\xi$ and $\xi'$, in particular the one depicted in~\eqref{diag:first definition of diagonal action}. From this we construct the cube
\[
\resizebox{.3\textwidth}{!}
{\xymatrix@!0@=4em{D\flat P 
\ar[rr]^-{e\flat k}
\ar[dd]_-{\xi}
\ar[rd]|-{\delta\flat\rho}
&&
A\flat A
\ar[dd]_(.25){\chi_A}|-{\hole}
\ar[rd]^-{\alpha\flat\alpha}\\
&
D'\flat P'
\ar[rr]_(.33){e'\flat k'}
\ar[dd]_(.33){\xi'}
&&
A'\flat A'
\ar[dd]^-{\chi_{A'}}
\\
P
\ar@{{ |>}->}[rr]_(.25){k'}|-{\hole}
\ar[rd]_-{\rho}
&&
A
\ar[rd]^-{\alpha}
\\
&
P'
\ar@{{ |>}->}[rr]_-{k'}
&&
A'.}}
\]
We want to prove that the face on the left commutes. Since we already know that every other face commutes, this follows from the fact that $k'$ is a monomorphism.
\end{proof}

\section{Construction of the non-abelian tensor product}\label{Section Non-Abelian Tensor}

\subsection{The case of groups.}\label{subsection:groups case}

First of all, let us examine what happens in the category $\Grp$. The aim of this subsection is to explain how to obtain the non-abelian tensor product of two coterminal crossed modules of groups, without passing through set-theoretical constructions.

Let $M$ and $N$ be two groups acting on each other via $\xi^M_N\colon M\flat N\to N$ and $\xi^N_M\colon {N\flat M\to M}$. Denote ${}^mn$ the action of $m\in M$ on $n\in N$, and ${}^nm$ the action of $n\in N$ on $m\in M$.

\begin{definition}[\cite{BL87}]
\label{defi:non-abelian tensor product of groups}
Given two groups $M$ and $N$ acting on each other (and on themselves by conjugation) we define their \emph{non-abelian tensor product} $M\otimes N$ as the group generated by the symbols $m\otimes n$ for $m\in M$ and $n\in N$, subject to the relations
\[
(mm')\otimes n=({}^mm'\otimes {}^mn)(m\otimes n)
\qquad\qquad
m\otimes (nn')=(m\otimes n)({}^nm\otimes {}^nn')
\]
for all $m$, $m'\in M$ and $n$, $n'\in N$.
\end{definition}

Although the above definition works for arbitrary actions, the main results of~\cite{BL87} that we are interested in, always assume that those actions are \emph{compatible} in a precise sense. Such a pair of compatible actions $(\xi^M_N,\xi^N_M)$ is equivalent to the datum of a third object $L$ and two crossed module structures $\mu\colon M\to L$ and $\nu\colon N\to L$; by~\cite{dMVdL19}, this is true both in the case of groups and in the general semi-abelian setting. For the sake of simplicity, from now on we will do the same: we shall always deal with non-abelian tensor products of pairs of compatible actions, and we shall always assume that those actions are induced by a pair of coterminal crossed modules. In particular, we formulate all definitions and results in terms of crossed modules. For instance:

\begin{remark}
The tensor product of two $L$-crossed modules carries a natural $L$-crossed module structure. Thus it may be seen as a bifunctor
\[
\otimes\colon {\XMod_L(\Grp )\times \XMod_L(\Grp )\to \XMod_L(\Grp )}.
\]
\end{remark}

The non-abelian tensor product satisfies a universal property which determines it: let us recall Proposition~2.15 from~\cite{BL87}.

\begin{proposition}\label{prop:prop2.15 in BL87}
Let $(M\xrightarrow{\mu}L,\xi_M)$ and $(N\xrightarrow{\nu}L,\xi_N)$ be crossed modules, so that $M$ and $N$ act on both $M$ and $N$ via $L$. Then there is a crossed square as on the left
\[
\xymatrix{
M\otimes N \ar[r]^-{\pi_M} \ar[d]_-{\pi_N} & M \ar[d]^-{\mu}\\
N \ar[r]_-{\nu} & L
}
\qquad\qquad
\xymatrix{
P \ar[r]^-{p_M} \ar[d]_-{p_N} & M \ar[d]^-{\mu}\\
N \ar[r]_-{\nu} & L
}
\]
where $\pi_M(m\otimes n)=m{}^{n}m^{-1}$, $\pi_N(m\otimes n)={}^{m}nn^{-1}$ and $h(m,n)=m\otimes n$. This crossed square is universal in the sense that it satisfies the following two equivalent conditions:
\begin{tfae}
\item If the square on the right is another crossed square (with the same $\mu$ and $\nu$), then there is a unique morphism of crossed squares $\morph{\phi & 1_M}{1_N & 1_L}$ from the left-hand to the right-hand crossed square which is the identity on $M$, $N$ and $L$ and where $\phi\colon {M\otimes N\to P}$.
\item The diagram of crossed squares in Figure~\ref{Fig:pushout of crossed square in Prop 2.15} 
\begin{figure}

\resizebox{.45\textwidth}{!}
{
\xymatrix@=1em{
0 \ar[rr] \ar[dd] && 0 \ar[dd] &&&& 0 \ar[rr] \ar[dd] && M \ar[dd]^-{\mu}\\
&&& \ar[rr]^-{\morph{1_0 & 0}{1_0 & 1_L}} && \\
0 \ar[rr] && L &&&& 0 \ar[rr] && L \\
& \ar[dd]_-{\morph{1_0 & 1_0}{0 & 1_L}} &&&&&& \ar[dd]^-{\morph{0 & 1_M}{0 & 1_L}} \\
\\
& &&&& && \\
0 \ar[rr] \ar[dd] && 0 \ar[dd] &&&& M\otimes N \ar[rr]^-{\pi^{M\otimes N}_M} \ar[dd]_-{\pi^{M\otimes N}_N} && M \ar[dd]^-{\mu}\\
&&& \ar[rr]_-{\morph{0 & 0}{1_N & 1_L}} && \\
N \ar[rr]_-{\nu} && L &&&& N \ar[rr]_-{\nu} && L \\
}}
\caption{A pushout of crossed squares}\label{Fig:pushout of crossed square in Prop 2.15}
\end{figure}
is a pushout in $\XSqr(\Grp)$.
\end{tfae}
\end{proposition}

We can reinterpret this result as a way to construct the non-abelian tensor product $M\otimes N$---namely, as the upper-left group in the pushout crossed square of  Figure~\ref{Fig:pushout of crossed square in Prop 2.15}. This process does not involve generators and relations and hence completely avoids the use of set-theoretical tools. In order to generalise this construction to $\XSqr(\A)$ we need the description in Section~\ref{Section:Double groupoids and double reflexive graphs} of pushouts of this kind in the category $\Grpd^2(\A)\simeq \XSqr(\A)$. Again, here and in what follows, $\A$ is a semi-abelian category that satisfies \SH.

\begin{construction}\label{Construction Tensor}
In a semi-abelian category with \SH, consider internal $L$-crossed modules $(M\xrightarrow{\mu}L,\xi_M)$ and ${(N\xrightarrow{\nu}L,\xi_N)}$ and their induced internal groupoid structures
\begin{align*}
\label{diag:two coterminal groupoid structures}
\vcenter{
\xymatrix{
 N \ar@{{ |>}->}[r]^-{k_N} & N\rtimes L \ar@<1ex>[r]^-{d_N} \ar@<-1ex>[r]_-{c_N} & L \ar[r]|-{e_M} \ar[l]|-{e_N} & M\rtimes L \ar@<1ex>[l]^-{d_M} \ar@<-1ex>[l]_-{c_M} & M.\ar@{{ |>}->}[l]_-{k_M}
}
}
\end{align*}
In $\Grpd^2(\A)$, we construct the span of Figure~\ref{Fig:induced span of double groupoids}; 
\begin{figure}

\resizebox{.75\textwidth}{!}{
\xymatrix{
C \ar@<1ex>[rr]^-{d_D} \ar@<-1ex>[rr]_-{c_D} \ar@<1ex>[dd]^-{1} \ar@<-1ex>[dd]_-{1} & & D \ar[ll]|-{e_D} \ar@<1ex>[dd]^-{1} \ar@<-1ex>[dd]_-{1} && D \ar@<1ex>[dd]^-{1} \ar@<-1ex>[dd]_-{1} \ar@<1ex>[rr]^-{1} \ar@<-1ex>[rr]_-{1} && D \ar[ll]|-{1} \ar@<1ex>[dd]^-{1} \ar@<-1ex>[dd]_-{1} && B \ar@<1ex>[dd]^-{d_R} \ar@<-1ex>[dd]_-{c_R} \ar@<1ex>[rr]^-{1} \ar@<-1ex>[rr]_-{1} && B \ar[ll]|-{1} \ar@<1ex>[dd]^-{d_R} \ar@<-1ex>[dd]_-{c_R}\\
&& \ar@{}[rr]^(.25){}="c"^(.75){}="d" \ar "d";"c"_-{\morph{e_D& 1_D}{e_D& 1_D}} && && \ar@{}[rr]^(.25){}="a"^(.75){}="b" \ar "a";"b"^-{\morph{e_R & e_R}{1_D & 1_D}} && \\
C \ar[uu]|-{1} \ar@<1ex>[rr]^-{d_D} \ar@<-1ex>[rr]_-{c_D} && D \ar[ll]|-{e_D} \ar[uu]|-{1} && D \ar[uu]|-{1} \ar@<1ex>[rr]^-{1} \ar@<-1ex>[rr]_-{1} & & D \ar[uu]|-{1} \ar[ll]|-{1} && D \ar[uu]|-{e_R} \ar@<1ex>[rr]^-{1} \ar@<-1ex>[rr]_-{1} && D \ar[uu]|-{e_R} \ar[ll]|-{1}
}
}
\caption{The span induced by a pair of groupoids}\label{Fig:induced span of double groupoids}
\end{figure}
in order to compute its pushout, we use the fact that the reflector preserves colimits. This means that we see it as a diagram in $\RG^2(\A)$ and compute its pushout there, via the pushout in $\A$ depicted in Figure~\ref{Fig:pushout in RG^2(A)}. 
\begin{figure}
$\xymatrix@!0@=3em{
Q \pullbackcorner \ar@<1ex>[dd]^-{d_L} \ar@<-1ex>[dd]_-{c_L} \ar@<1ex>[rr]^-{d_U} \ar@<-1ex>[rr]_-{c_U} && M\rtimes L \ar[ll]|-{e_U} \ar@<1ex>[dd]^-{d_M} \ar@<-1ex>[dd]_-{c_M}\\
\\
N\rtimes L \ar[uu]|-{e_L} \ar@<1ex>[rr]^-{d_N} \ar@<-1ex>[rr]_-{c_N} && L \ar[uu]|-{e_M} \ar[ll]|-{e_N}
}$
\caption{Pushout in $\RG(\A)$ computed in $\A$}\label{Fig:pushout in RG^2(A)}
\end{figure}
In other words, $Q\coloneq(M\rtimes L)+_L(N\rtimes L)$ is the pushout of $e_M$ along $e_N$, and the maps $d_U$, $c_U$, $d_L$ and $c_L$ are defined via the universal property of the pushout:
\begin{align*}
d_U\coloneq\pushout{1_{M\rtimes L}}{e_M\comp d_N} \qquad d_L\coloneq \pushout{e_N\comp d_M}{1_{N\rtimes L}}\qquad
c_U\coloneq\pushout{1_{M\rtimes L}}{e_M\comp c_N} \qquad c_L\coloneq \pushout{e_N\comp c_M}{1_{N\rtimes L}}.
\end{align*}
Applying the left adjoint to Figure~\ref{Fig:pushout in RG^2(A)} we obtain the desired double groupoid in Figure~\ref{Fig:CrSqWithTensor}, indeed the pushout in $\Grpd^2(\A)$ of the span in Figure~\ref{Fig:induced span of double groupoids}. Note that $Q_{M\otimes N}$ is given by
\begin{equation}\label{Q}
\frac{(M\rtimes L)+_L(N\rtimes L)}{[K_{d_L},K_{c_L}]\vee[K_{d_U},K_{c_U}]}.	
\end{equation}
By normalising this double groupoid (that is, by computing the kernels of the \lq\lq domain\rq\rq\ morphisms and of the induced maps), we go back from $\Grpd^2(\A)$ to $\XSqr(\A)$ obtaining the internal crossed square in Figure~\ref{Fig:CrSqWithTensor}. 
\begin{figure}

\resizebox{.35\textwidth}{!}
{\xymatrix@!0@=6em{
M\otimes N \pullbackcorner \ar@{{ |>}->}[r] \ar@{{ |>}->}[d] & K_{\overline{d_L}} \ar@<1ex>[r] \ar@<-1ex>[r] \ar@{{ |>}->}[d] & M \ar[l] \ar@{{ |>}->}[d]\\
K_{\overline{d_U}} \ar@<1ex>[d] \ar@<-1ex>[d] \ar@{{ |>}->}[r] & Q_{M\otimes N} \ar@<1ex>[d]^-{\overline{d_L}} \ar@<-1ex>[d]_-{\overline{c_L}} \ar@<1ex>[r]^-{\overline{d_U}} \ar@<-1ex>[r]_-{\overline{c_U}} & M\rtimes L \ar[l]|-{\overline{e_U}} \ar@<1ex>[d]^-{d_M} \ar@<-1ex>[d]_-{c_M}\\
N \ar[u] \ar@{{ |>}->}[r] & N\rtimes L \ar[u]|-{\overline{e_L}} \ar@<1ex>[r]^-{d_N} \ar@<-1ex>[r]_-{c_N} & L \ar[u]|-{e_M} \ar[l]|-{e_N}
}}
\caption{Crossed square involving tensor}\label{Fig:CrSqWithTensor}
\end{figure}
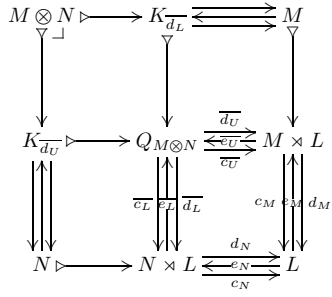
Using the equivalence $\XSqr(\A)\simeq\Grpd^2(\A)$ we now have that this crossed square is the pushout in $\XSqr(\A)$ depicted in Figure~\ref{Fig:pushout of crossed square in Prop 2.15}.
\end{construction}

\begin{definition}
\label{defi:abstract non-abelian tensor product}
Given internal $L$-crossed modules $(M\xrightarrow{\mu}L,\xi_M)$ and $({N\xrightarrow{\nu}L},\xi_N)$ we define their \emph{non-abelian tensor product} $M\otimes N$ as the top left object in the square
\[
\xymatrix{
M\otimes N \ar[r]^-{\pi^{M\otimes N}_M} \ar[d]_-{\pi^{M\otimes N}_N} \ar[rd]|-{\lambda} & M \ar[d]^-{\mu}\\
N \ar[r]_-{\nu} & L
}
\]
constructed above.
\end{definition}

\begin{proposition}
The non-abelian tensor product $M\otimes N$ has an internal $L$-crossed module structure, namely $(M\otimes N\xrightarrow{\lambda}L,\xi)$, where the action $\xi$ is defined as in~\ref{rmk:definitions of diagonal action}.
\end{proposition}
\begin{proof}
This follows immediately from Proposition~\ref{prop:the diagonal is a crossed module}.
\end{proof}

\begin{proposition}\label{Proto Tensor Functor}
Consider two $L$-crossed modules $(M\xrightarrow{\mu}L,\xi^L_M)$, $(N\xrightarrow{\nu}L,\xi^L_N)$, two $L'$-crossed modules $(M'\xrightarrow{\mu'}L',\xi'^{L'}_{M'})$, $(N'\xrightarrow{\nu'}L',\xi'^{L'}_{N'})$ and two morphisms of internal crossed modules
\begin{align*}
(M\xrightarrow{\mu}L,\xi^L_M)\xlongrightarrow{(f,l)} (M'\xrightarrow{\mu'}L',\xi'^{L'}_{M'}),\qquad
(N\xrightarrow{\nu}L,\xi^L_N)\xlongrightarrow{(g,l)} (N'\xrightarrow{\nu'}L',\xi'^{L'}_{N'}).
\end{align*}
Then there exists a unique morphism $f\otimes g\colon M\otimes N\to M'\otimes N'$ such that $\morph{f\otimes g & f}{g & l}$ is a morphism of internal crossed squares.
\end{proposition}
\begin{proof}
Consider Figure~\ref{Fig:Functoriality of the tensor product}, 
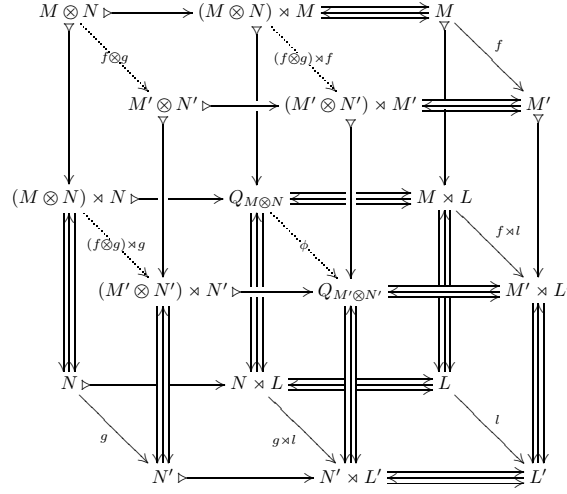
\begin{figure}

\resizebox{.6\textwidth}{!}{
\xymatrix@!0@=5em{
M\otimes N 
\ar@{{ |>}->}[rr]
\ar@{{ |>}->}[dd]
\ar@{.>}[rd]|-{f\otimes g}
&&
(M\otimes N)\rtimes M
\ar@<3pt>[rr]
\ar@<-3pt>[rr]
\ar@{{ |>}->}[dd]|-{\hole}
\ar@{.>}[rd]|-{(f\otimes g)\rtimes f}
&&
M
\ar[ll]
\ar@{{ |>}->}[dd]|-{\hole}
\ar[rd]^-{f}
\\
&
M'\otimes N'
\ar@{{ |>}->}[rr]
\ar@{{ |>}->}[dd]
&&
(M'\otimes N')\rtimes M'
\ar@<3pt>[rr]
\ar@<-3pt>[rr]
\ar@{{ |>}->}[dd]
&&
M'
\ar[ll]
\ar@{{ |>}->}[dd]
\\
(M\otimes N)\rtimes N
\ar@<3pt>[dd]
\ar@<-3pt>[dd]
\ar@{{ |>}->}[rr]|(.50){\hole}
\ar@{.>}[rd]|-{(f\otimes g)\rtimes g}
&&
Q_{M\otimes N}
\ar@<3pt>[dd]|(.50){\hole}
\ar@<-3pt>[dd]|(.50){\hole}
\ar@<3pt>[rr]|(.50){\hole}
\ar@<-3pt>[rr]|(.50){\hole}
\ar@{.>}[rd]|-{\phi}
&&
M\rtimes L
\ar@<3pt>[dd]|(.50){\hole}
\ar@<-3pt>[dd]|(.50){\hole}
\ar[ll]|(.50){\hole}
\ar[rd]^-{f\rtimes l}
\\
&
(M'\otimes N')\rtimes N'
\ar@<3pt>[dd]
\ar@<-3pt>[dd]
\ar@{{ |>}->}[rr]
&&
Q_{M'\otimes N'}
\ar@<3pt>[dd]
\ar@<-3pt>[dd]
\ar@<3pt>[rr]
\ar@<-3pt>[rr]
&&
M'\rtimes L'
\ar@<3pt>[dd]
\ar@<-3pt>[dd]
\ar[ll]
\\
N
\ar[uu]
\ar@{{ |>}->}[rr]|(.50){\hole}
\ar[rd]_-{g}
&&
N\rtimes L
\ar[uu]|(.50){\hole}
\ar[rd]_-{g\rtimes l}
\ar@<3pt>[rr]|(.50){\hole}
\ar@<-3pt>[rr]|(.50){\hole}
&&
L
\ar[uu]|(.50){\hole}
\ar[ll]|(.50){\hole}
\ar[rd]^-{l}
\\
&
N'
\ar[uu]
\ar@{{ |>}->}[rr]
&&
N'\rtimes L'
\ar[uu]
\ar@<3pt>[rr]
\ar@<-3pt>[rr]
&&
L'
\ar[uu]
\ar[ll]
}}
\caption{Functoriality of the tensor product}\label{Fig:Functoriality of the tensor product}
\end{figure}
where $\phi$ is determined by the universal property of the diagram in Figure~\ref{Fig:CrSqWithTensor}: in particular $\phi$ is the only morphism which makes $\morph{\phi & f\rtimes l}{g\rtimes l & l }$ a morphism of double groupoids. Since the other dotted maps are uniquely induced by taking kernels, $f\otimes g$ is automatically the unique morphism such that $\morph{f\otimes g &f}{g& l}$ is a morphism of internal crossed squares.
\end{proof}

\begin{corollary}
\label{cor:the non-abelian tensor product is a bifunctor}
In the situation of Proposition~\ref{Proto Tensor Functor},
\[
(M\otimes N\xrightarrow{\lambda}L,\xi)\xrightarrow{(f\otimes g,l)}(M'\otimes N'\xrightarrow{\lambda'}L',\xi')
\]
is a morphism of internal crossed modules. Hence the non-abelian tensor product is a bifunctor $\otimes\colon \XMod_L(\A)\times \XMod_L(\A)\to \XMod_L(\A)$.
\end{corollary}
\begin{proof}
The first result applies Proposition~\ref{prop:every morphism of internal crossed squares has an underlying morphism of internal crossed modules between the diagonals} to the morphism $\morph{f\otimes g& f}{g& l}$. The second part is the particular case where $l=1_L$.
\end{proof}

The tensor product operation is obviously commutative, up to isomorphism, by construction; but it is not associative---see~\cite{Ell91} for an argument in the case of Lie algebras; see also Section~\ref{Sec:Examples}.

\begin{example}
\label{ex:tensor product with trivial crossed module}
Consider the two crossed modules $(N\xrightarrow{\nu}L,\xi^L_N)$ and $(0\xrightarrow{0}L,\tau^L_0)$. Let us compute their non-abelian tensor product. The induced internal groupoids are given in the diagram
\[
\xymatrixcolsep{3pc}
\xymatrixrowsep{3pc}
\xymatrix{
N \ar@{{ |>}->}[r]^-{k_N} & N\rtimes L \ar@<1ex>[r]^-{d_N} \ar@<-1ex>[r]_-{c_N} & L \ar[r]|-{1_L} \ar[l]|-{e_N} & L \ar@<1ex>[l]^-{1_L} \ar@<-1ex>[l]_-{1_L} & 0 \ar@{{ |>}->}[l]_-{0}
}
\]
The double groupoid in Figure~\ref{Fig:pushout in RG^2(A)} has $M\rtimes L=L$, which is easily seen to imply $0\otimes N\cong 0$.
\end{example}

The tensor product may be viewed (or defined) as an initial object in a certain category of crossed squares.

\begin{proposition}
\label{prop:universal property of non-abelian tensor product}
Consider an internal crossed square as on the left
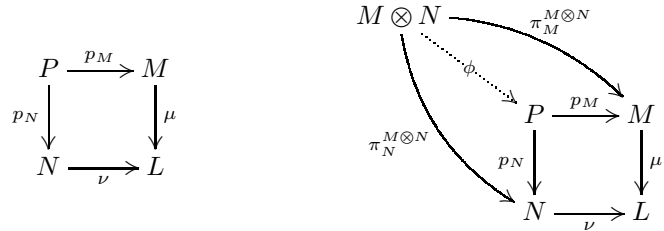
\[
\vcenter{\xymatrix{
P \ar[r]^-{p_M} \ar[d]_-{p_N} & M \ar[d]^-{\mu}\\
N \ar[r]_-{\nu} & L
}}
\qquad\qquad\qquad
\vcenter{\xymatrix{
M\otimes N \ar@/^1pc/[rrd]^-{\pi^{M\otimes N}_M} \ar@/_1pc/[rdd]_-{\pi^{M\otimes N}_N} \ar@{.>}[rd]|-{\phi}\\
& P \ar[r] ^-{p_M}\ar[d]_-{p_N} & M \ar[d]^-{\mu}\\
& N \ar[r]_-{\nu} & L
}}
\]
Then there exists a unique $\phi$ such that the diagram on the right commutes, making $\morph{\phi & 1_M}{1_N& 1_L}$ a morphism of internal crossed squares.
\end{proposition}
\begin{proof}
We first shift to the internal double groupoid setting and construct the diagram in Figure~\ref{Fig:universal property of non-abelian tensor product}. 
\begin{figure}

\resizebox{.5\textwidth}{!}{
\xymatrix@!0@=4.5em{
M\otimes N 
\ar@{{ |>}->}[rr]
\ar@{{ |>}->}[dd]
\ar@{.>}[rd]|-{\phi}
&&
(M\otimes N)\rtimes M
\ar@<3pt>[rr]
\ar@<-3pt>[rr]
\ar@{{ |>}->}[dd]|-{\hole}
\ar@{.>}[rd]|-{\phi\rtimes 1_M}
&&
M
\ar[ll]
\ar@{{ |>}->}[dd]|-{\hole}
\ar@{=}[rd]
\\
&
P
\ar@{{ |>}->}[rr]
\ar@{{ |>}->}[dd]
&&
P\rtimes M
\ar@<3pt>[rr]
\ar@<-3pt>[rr]
\ar@{{ |>}->}[dd]
&&
M
\ar[ll]
\ar@{{ |>}->}[dd]
\\
(M\otimes N)\rtimes N
\ar@<3pt>[dd]
\ar@<-3pt>[dd]
\ar@{{ |>}->}[rr]|(.50){\hole}
\ar@{.>}[rd]|-{\phi\rtimes 1_N}
&&
Q_{M\otimes N}
\ar@<3pt>[dd]|(.50){\hole}
\ar@<-3pt>[dd]|(.50){\hole}
\ar@<3pt>[rr]|(.50){\hole}
\ar@<-3pt>[rr]|(.50){\hole}
\ar@{.>}[rd]|-{\phi_0}
&&
M\rtimes L
\ar@<3pt>[dd]|(.50){\hole}
\ar@<-3pt>[dd]|(.50){\hole}
\ar[ll]|(.50){\hole}
\ar@{=}[rd]
\\
&
P\rtimes N
\ar@<3pt>[dd]
\ar@<-3pt>[dd]
\ar@{{ |>}->}[rr]
&&
P_0
\ar@<3pt>[dd]
\ar@<-3pt>[dd]
\ar@<3pt>[rr]
\ar@<-3pt>[rr]
&&
M\rtimes L
\ar@<3pt>[dd]
\ar@<-3pt>[dd]
\ar[ll]
\\
N
\ar[uu]
\ar@{{ |>}->}[rr]|(.50){\hole}
\ar@{=}[rd]
&&
N\rtimes L
\ar[uu]|(.50){\hole}
\ar@{=}[rd]
\ar@<3pt>[rr]|(.50){\hole}
\ar@<-3pt>[rr]|(.50){\hole}
&&
L
\ar[uu]|(.50){\hole}
\ar[ll]|(.50){\hole}
\ar@{=}[rd]
\\
&
N
\ar[uu]
\ar@{{ |>}->}[rr]
&&
N\rtimes L
\ar[uu]
\ar@<3pt>[rr]
\ar@<-3pt>[rr]
&&
L
\ar[uu]
\ar[ll]
}}

\caption{The universal property of the non-abelian tensor product}\label{Fig:universal property of non-abelian tensor product}
\end{figure}
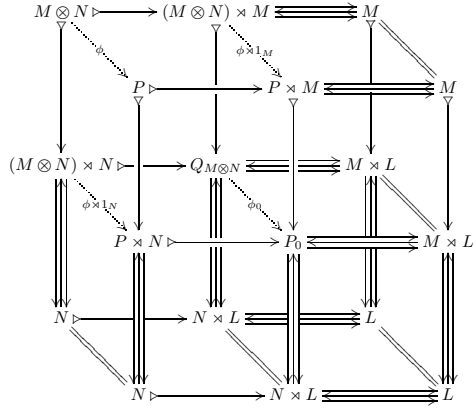
Here $\phi_0$ is induced by the fact that the double groupoid involving $Q_{M\otimes N}$ is defined as a pushout in $\Grpd^2(\A)$, whereas the maps $\phi\rtimes 1_M$ and $\phi\rtimes 1_N$ are the maps induced between the kernels and finally $\phi$ is given by the front square in the upper-left cube being a pullback. The fact that $\phi$ is the unique morphism making $\morph{\phi & 1_M}{1_N & 1_L}$ a~morphism of crossed squares comes from the fact that $\phi_0$ is the unique morphism such that $\morph{\phi_0 & 1_{M\rtimes L}}{1_{N\rtimes L} & 1_L}$ is a morphism of double groupoids.
\end{proof}

\section{Examples of the tensor product}\label{Sec:Examples}
In this section we consider three different types of examples of the tensor product: first we look at the case of two normal subobjects, viewed as crossed modules \eqref{SubsecNormalSubobjects}; then we explore the other end of the spectrum: pairs of abelian objects acting trivially upon one another \eqref{SubSecAbelian}; finally in \eqref{SubsecLie} we treat the non-abelian tensor product of Lie algebras.
 
We shall work in the context of an algebraically coherent~\cite{CGVdL15b} semi-abelian category $\A$. This means that the natural comparison morphism $\binom{1_X\flat \iota_Y}{1_X\flat \iota_Z}\colon X\flat Y+ X\flat Z \to X\flat (Y+Z)$ is a regular epimorphism, for each choice of $X$, $Y$, $Z\in \A$---recall the definition of $\flat$ given in~\eqref{Def Flat}. All \emph{locally algebraically cartesian closed} semi-abelian categories~\cite{Gray2012} are examples, since then by definition, the comparison morphisms $\binom{1_X\flat \iota_Y}{1_X\flat \iota_Z}$ are isomorphisms. We find groups, Lie algebras, crossed modules, cocommutative Hopf algebras. Next, all \emph{Orzech categories of interest}~\cite{Orz72} are algebraically coherent. 

All algebraically coherent semi-abelian categories satisfy \SH. More precisely, when\-ever $M$, $N\lhd L$ in an algebraically coherent semi-abelian category, the so-called \emph{Three Subobjects Lemma} $[M,N,L]=[M,[N,L]]\vee [N,[M,L]]$ implies that  $[M,N,L]\leq [M,N]$. Further convenient properties of algebraically coherent categories will be recalled below.

\subsection{The Smith-Pedicchio commutator}
We start with something which is not quite an example, but very close to being one. Pedicchio's categorical approach to the Smith commutator of equivalence relations (see~\cite{Ped95b,Smi76,BB04}) involves a double equivalence relation as in Figure~\ref{Figure Delta}. 
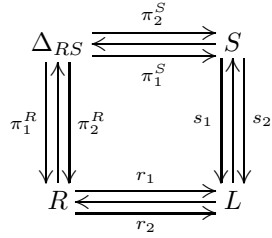
\begin{figure}
\[
\vcenter{\xymatrix@=2em{\Delta_{RS} \ar@<1ex>[rr]^-{\pi_2^S} \ar@<-1ex>[rr]_-{\pi_1^S}\ar@<1ex>[dd]^-{\pi_2^R}\ar@<-1ex>[dd]_-{\pi_1^R}&&S\ar[ll]\ar@<1ex>[dd]^-{s_2}\ar@<-1ex>[dd]_-{s_1}\\\\
R \ar[uu]\ar@<1ex>[rr]^-{r_1}\ar@<-1ex>[rr]_-{r_2}&& L\ar[ll]\ar[uu]}}
\]
\caption{The universal double equivalence relation $\Delta_{RS}$ over $R$ and $S$}\label{Figure Delta}
\end{figure}
Given equivalence relations $R$ and $S$ on an object $
L$, the double equivalence relation $\Delta_{RS}$ is initial amongst all double equivalence relations over $R$ and~$S$. (In the words of~\cite{Ped95b}, it is the smallest such.) Thus, it satisfies part, but only part, of the universal property depicted in Figure~\ref{Fig:universal property of non-abelian tensor product}: it is initial among equivalence relations rather than initial among double groupoids. 

As explained to us by Cyrille Simeu, it is not hard to see that the corresponding crossed square (where $M$ and $N\lhd L$ are the respective normalisations of $R$ and~$S$ and all arrows are normal monomorphisms) has an upper left corner which is precisely the so-called \emph{Ursini commutator} $[M,N]^{\mathcal{U}}_L$ of $M$ and $N$ in the sense of \cite{Mantovani:Ursini}. By the results of \cite{HVdL11}, in the present, algebraically coherent context, this commutator coincides with~$[M,N]$. So 
\[
\xymatrix{[M,N]  \ar@{{ |>}->}[r]  \ar@{{ |>}->}[d] & N \ar@{{ |>}->}[d]^-n\\
M \ar@{{ |>}->}[r]_-m & L}
\]
is the initial crossed square \emph{of normal monomorphisms} over $m$ and $n$.

\subsection{Normal subobjects.}\label{SubsecNormalSubobjects}
We may now ask ourselves what is the tensor product of two normal subobjects $M$, $N\lhd L$ in $\A$. From the above it is not hard to deduce that this tensor product must have $[M,N]$ as a quotient. Let us give an alternative and slightly more explicit explanation for this fact.

First of all, we know from Example~\ref{Example intersection} that the intersection of $M$ and $N$ is part of a crossed square, so we have a canonical map $h\colon M\otimes N\to M\wedge N$. Now in the computation of the tensor product (the construction in~\ref{Construction Tensor}, with in particular equation~\eqref{Q}), we have to take (a quotient of) a certain pushout in~$\A$, which happens to be the sum over $L$ of the respective denormalisations $R$ and $S$ of~$M$ and~$N$; the intersection of the kernels $K_{\overline{d_L}}$ and $K_{\overline{d_U}}$ of the first projections $\overline{d_L}$ and $\overline{d_U}$ in this double reflexive graph (also called the \emph{direction} of the square of first projections $d_N\comp \overline{d_L}=d_M\comp \overline{d_U}$) is the underlying object of the cosmash product $R\diamond_L S$ of $R$ and $S$ over~$L$, which (by Lemma~2.9 in~\cite{SVdL3}) in an algebraically coherent setting is nothing but the cosmash product $M\diamond N$ of $M$ and $N$. The tensor product $M\otimes N$ is a quotient of this object, and~$h$ composed with the quotient map is the canonical map $M\diamond N\to M\wedge N$. 

Thus we see that the image of $h\colon M\otimes N\to M\wedge N$ is $[M,N]\leq M\wedge N$. In particular, $h$ is far from being a regular epimorphism in general.

\subsection{Abelian objects acting trivially on one another.}\label{SubSecAbelian}
Recall that $\Nil_2(\A)$ is the full subcategory of $\A$ of the objects $A$ where $[A,A,A]$ is trivial~\cite{HVdL11,CGVdL15b}: this is the Birkhoff subcategory whose reflector $\nil_2\colon\A\to \Nil_2(\A)$ sends $A$ to ${A}/{[A,A,A]}$. When $\A$ is algebraically coherent, $[A,A,A]=[[A,A],A]\lhd A$. We are going to prove:

\begin{theorem}\label{Theorem Nil2}
When $\A$ is an algebraically coherent semi-abelian category, for any pair of abelian objects $M$ and $N$ acting trivially on one another we have $M\otimes N\cong M\diamond_2N$, where $M\diamond_2N$ is the cosmash product of $M$ and $N$ in the 2-nilpotent core $\Nil_2(\A)$ of~$\A$.
\end{theorem}

Via \cite{MacHenry} this allows us to recover the result from \cite{BL87} that when $M$ and~$N$ are groups, $M\otimes N=M\otimes_{\Z}N$. This also exhibits the \emph{bilinear product} of \cite{HVdL2} as a special case of the non-abelian tensor product.

\begin{proof}[of Theorem~\ref{Theorem Nil2}]
The construction of the tensor product tells us that we should consider the crossed modules $M\to 0$ and $N\to 0$, which correspond to the reflexive graphs $M$ and $N$ on the object $0$. Here we use that $M$ and $N$ are abelian as in Example~\ref{ex:crossed point associated to trivial crossed module}. Hence the tensor product is the normalisation of the internal groupoid obtained from the quotient of the sum $M+N$ by the join of commutators $J=[N\flat M,N\flat M]\vee [M\flat N, M\flat N]$ as in Figure~\ref{Fig:Abelian}. 
\begin{figure}
\resizebox{.3\textwidth}{!}{
\xymatrix@!0@=6em{M\otimes N \ar@{{ |>}->}[r] \ar@{{ |>}->}[d] & \frac{M\flat N}{J} \ar@<-1ex>[r] \ar@<1ex>[r] \ar@{{ |>}->}[d] & N \ar@{=}[d] \ar[l]\\
\frac{N\flat M}{J} \ar@<-1ex>[d] \ar@<1ex>[d] \ar@{{ |>}->}[r] & \frac{M+N}{J} \ar@<-1ex>[d] \ar@<1ex>[d] \ar@<-1ex>[r] \ar@<1ex>[r] & N \ar@<-1ex>[d] \ar@<1ex>[d] \ar[l] \\
M \ar@{=}[r] \ar[u] & M \ar@<-1ex>[r] \ar@<1ex>[r] \ar[u] & 0 \ar[l] \ar[u] }}
\caption{Tensor product of abelian objects}\label{Fig:Abelian}	
\end{figure}
Further note that 
\[
[M,N]=M\diamond N=(M\flat N)\wedge (N\flat M)\leq M+N.
\]
In fact, $M\otimes N\cong (M\diamond N)/J$, because indeed $J$ factors through both $M\flat N$ and $N\flat M$, as we shall see, and so it factors through their intersection $M\diamond N$ as well. Since, as $\nil_2(M+N)\cong \nil_2(M)+_2\nil_2(N)$, we have the $3\times 3$ diagram
\[
\resizebox{.6\textwidth}{!}{
\xymatrix@!0@R=5em@C=15em{[M+N,M+N,M+N] \ar@{{ |>}->}[r] \ar@{{ |>}->}[d] & [M+N,M+N,M+N] \ar@{-{ >>}}[r] \ar@{{ |>}->}[d] & 0 \ar@{{ |>}->}[d]\\
M\diamond N \ar@{{ |>}->}[r] \ar@{-{ >>}}[d] & M+N \ar@{-{ >>}}[r] \ar@{-{ >>}}[d] & M\times N \ar@{-{ >>}}[d]\\
M\diamond_2 N \ar@{{ |>}->}[r] & \nil_2(M)+_2\nil_2(N) \ar@{-{ >>}}[r] & M\times N,}}
\]
it now suffices to prove that $J$ coincides with $[M+N,M+N,M+N]$ in order to show that $M\otimes N\cong M\diamond_2N$. So let us compare the two. First of all, by algebraic coherence we can write $[M+N,M+N,M+N]=[[M+N,M+N],M+N]$. We use Proposition~\ref{Higgins properties} on the latter commutator. Since the kernel and the splitting in the split short exact sequence~\eqref{Def Flat} are jointly extremal-epimorphic, we have $M+N=(M\flat N)\vee M$; from the split short exact sequence
\[
\xymatrix@=3em{0 \ar[r]& N\diamond M\ar@{{ |>}->}[r] & M\flat N \ar@<.5ex>@{-{ >>}}[r]^{\binom{0}{1_N}\circ k_{M,N}} & N \ar[r] \ar@<.5ex>[l]^-{\eta^M_N} & 0}
\]
we deduce $M\flat N=(N\diamond M)\vee N$; whence
\begin{align*}
	[M+N,M+N]&=[M\flat N, M+N]\vee [M,M+N]\vee [M\flat N,M,M+N]\\
	&= [M\flat N, M\flat N]\vee [M\flat N, M]\vee[M\flat N, M\flat N, M]\\
	&= [M\flat N, M\flat N]\vee [M\flat N, M]\\
	&= [M\flat N, M\flat N]\vee [[N,M]\vee N,M]\\
	&= [M\flat N, M\flat N]\vee [[N,M],M]\vee [N,M]\vee [[N,M],N,M]\\
	&= [M\flat N, M\flat N]\vee [N,M].
\end{align*}
Note that
\begin{align*}
	[[N,M],M+N] &= [[N,M],M\flat N\vee N\flat M]\\
	&= [[N,M],M\flat N]\vee [[N,M],N\flat M]\vee [[N,M],M\flat N,N\flat M]\\
	&\leq [M\flat N,M\flat N]\vee [N\flat M,N\flat M]\vee [M\flat N,M\flat N,M+N]\\
	&= [M\flat N,M\flat N]\vee [N\flat M,N\flat M],
\end{align*}
so that
\begin{align*}
	&[[M+N,M+N],M+N]\\
	&=[[M\flat N, M\flat N]\vee [N, M],M+N]\\
	&=[[M\flat N, M\flat N], M+N]\vee [[N, M],M+N]\vee [[M\flat N, M\flat N], [N, M], M+N]\\
	&\leq[M\flat N, M\flat N]\vee [N\flat M,N\flat M] \vee [M\flat N, M\flat N, M+N]\\
	&=[M\flat N, M\flat N]\vee [N\flat M,N\flat M].
\end{align*}
Conversely, $[N\flat M,N\flat M] = [[N,M]\vee M, N\flat M]$ is
\begin{align*}
	 [[N,M], N\flat M]\vee [M, N\flat M]\vee [[N,M], M, N\flat M],
\end{align*}
whose terms are all contained in $[M+N,M+N,M+N]$, because 
\begin{align*}
	 [M, N\flat M] = [M, [N,M]\vee M]
	 = [M, [N,M]]\vee [M, M] \vee [M, [N,M],M]
\end{align*}
and $M$ is abelian. 
\end{proof}

\subsection{Lie algebras.}\label{SubsecLie}

The aim of this subsection is to show that the non-abelian tensor product of Lie algebras defined in~\cite{Ell91} coincides with the general definition of non-abelian tensor product when $\A=\LieAlg$, for any given commutative ring $R$. In order to do that we need to recall some definitions and results from~\cite{Ell91,CL98}.

From now on we will assume that $M$ and $N$ are two Lie algebras with crossed module structures on a common Lie algebra $L$, since in~\cite{dMVdL19} it is shown that this is the same as having two compatible actions of Lie algebras.

\begin{definition}[\cite{Ell91}]
\label{defi:non-abelian tensor product of Lie algebras}
Given two $R$-Lie algebras $M$ and $N$ acting on each other, their non-abelian tensor product $M\otimes_{\Lie} N$ is the Lie algebra generated by the symbols $m\otimes n$ with $m\in M$ and $n\in N$, subject to the relations
\begin{itemize}
\item $(\lambda m)\otimes n=\lambda(m\otimes n)=m\otimes (\lambda n)$,
\item $(m+m')\otimes n=m\otimes n+m'\otimes n$ and $m\otimes (n+n')=m\otimes n+m\otimes n'$,
\item $[m,m']\otimes n=m\otimes ({}^{m'}n)-m'\otimes ({}^{m}n)$ and $m\otimes [n,n']=({}^{n'}m)\otimes n-({}^{n}m)\otimes n'$,
\item $[m\otimes n,m'\otimes n']=-({}^{n}m)\otimes ({}^{m'}n')$,
\end{itemize}
for all $\lambda\in R$, $m$, $m'\in M$ and $n$, $n'\in N$.
\end{definition}

\begin{definition}[\cite{Ell91}]
\label{defi:Lie pairing}
Given two $R$-Lie algebras $M$ and $N$ acting on each other, and a third Lie algebra $P$, we say that a bilinear function $h\colon M\times N\to P$ is a \emph{Lie pairing} if
\begin{enumerate}
\item $h([m,m'],n)=h(m,{}^{m'}n)-h(m',{}^{m}n)$,
\item $h(m,[n,n'])=h({}^{n'}m,n)-h({}^{n}m,n')$,
\item $h({}^{n}m,{}^{m'}n')=-[h(m,n),h(m',n')]$,
\end{enumerate}
for all $m$, $m'\in M$ and $n$, $n'\in N$. The Lie pairing $h$ is said to be \emph{universal} if for any other Lie pairing $h'\colon M\times N\to P'$ there exists a unique Lie homomorphism $\phi\colon P\to P'$ such that $\phi\comp h=h'$.
\end{definition}

\begin{proposition}[Proposition~1 in~\cite{Ell91}]
\label{prop:prop 1 in Ell91}
Given two $R$-Lie algebras $M$ and $N$ acting on each other, $h\colon M\times N \to M\otimes_{\Lie} N\colon (m,n) \mapsto m\otimes n$ is a universal Lie pairing. 

Hence the non-abelian tensor product $M\otimes_{\Lie} N$ of two Lie algebras acting on each other is characterised (up to isomorphism) as the codomain of their universal Lie pairing.
\end{proposition}

\begin{definition}[\cite{Ell84,CL98}]
\label{defi:explicit definition of crossed square of Lie algebras}
A \emph{crossed square} in $\LieAlg$ is a commutative square of Lie algebras
\[
\xymatrix{
P \ar[r]^-{p_M} \ar[d]_-{p_N} & M \ar[d]^-{\mu}\\
N \ar[r]_-{\nu} & L
}
\]
endowed with Lie actions of $L$ on $P$, $M$ and $N$ (and hence Lie actions of $M$ on $N$ and~$P$ via $\mu$, and of $N$ on $M$ and $P$ via $\nu$) and a function $h\colon M\times N\to P$ such that
\begin{enumerate}[{\rm (X.1)}]
\setcounter{enumi}{-1}
\item $h$ is bilinear and satisfies
\[
\text{$h([m,m'],n)={}^{m}h(m',n)-{}^{m'}h(m,n)$, \quad $h(m,[n,n'])={}^{n}h(m,n')-{}^{n'}h(m,n)$;}
\]
 \item $p_M$ and $p_N$ are $L$-equivariant, and
 \[
 (M\xrightarrow{\mu}L,\xi_M),\qquad {(N\xrightarrow{\nu}L,\xi_N)}, \qquad(P\xrightarrow{\mu\circ p_M=\nu\circ p_N}L,\xi_P)
 \]
 are crossed modules;
 \item $p_M(h(m,n))=-{}^{n}m$ and $p_N(h(m,n))={}^{m}n$;
 \item $h(p_M(p),n)=-{}^{n}p$ and $h(m,p_N(p))={}^{m}p$;
 \item ${}^lh(m,n)=h({}^lm,n)+h(m,{}^ln)$;
\end{enumerate}
for all $l\in L$, $m$, $m'\in M$, $n$, $n'\in N$ and $p\in P$.
\end{definition}

\begin{lemma}[Theorem~30 in~\cite{CL98}]
\label{lemma:Lie crossed squares are the same as internal crossed squares in LieAlg}
Lie crossed squares, as just defined, coincide with internal crossed squares in the category $\LieAlg$.
\end{lemma}

\begin{lemma}[\cite{Khm99}]
\label{lemma:the non-abelian tensor product of Lie algebras gives us a crossed square}
For a pair of crossed modules $(M\xrightarrow{\mu}L,\xi_M)$ and $(N\xrightarrow{\nu}L,\xi_N)$, the square
\[
\xymatrix{
M\otimes_{\Lie}N \ar[r]^-{\rho_M} \ar[d]_-{\rho_N} & M \ar[d]^-{\mu}\\
N \ar[r]_-{\nu} & L
}
\]
in $\LieAlg$, with $\rho_M$ and $\rho_N$ defined via $\rho_M(m\otimes n)=-{}^{n}m$, $\rho_N(m\otimes n)={}^{m}n$
endowed with 
\begin{itemize}
 \item the actions $\xi_M$ and $\xi_N$,
 \item the action of $L$ on $M\otimes_{\Lie}N$ given by ${}^{l}(m\otimes n)\coloneq({}^{l}m)\otimes n + m\otimes ({}^{l}n)$,
\item the map $h\colon M\times N\to M\otimes N$ defined in Proposition~\ref{prop:prop 1 in Ell91},
\end{itemize}
is a crossed square (in the sense of Definition~\ref{defi:explicit definition of crossed square of Lie algebras}).
\end{lemma}

\begin{proposition}\label{Tensor for Lie algebras}
When $\A=\LieAlg$, the non-abelian tensor product $M\otimes N$ as in Definition~\ref{defi:abstract non-abelian tensor product} coincides with the tensor product of Lie algebras $M\otimes_{\Lie} N$ from Definition~\ref{defi:non-abelian tensor product of Lie algebras}.
\end{proposition}
\begin{proof}
The first step is to construct a Lie pairing from $M\times N$ to $M\otimes N$. We consider Figure~\ref{Fig:CrSqWithTensor} and denote with $j_M$ and $j_N$ the diagonal inclusions of $M$ and $N$ in $Q_{M\otimes N}$. We are going to define a function $h$ from $M\times N$ to $Q_{M\otimes N}$, and show that it factors through $M\otimes N$ as $h\colon M\times N\to M\otimes N$. Then we prove that it is a Lie pairing.

Since we are in $\LieAlg$ we can define $h$ directly on the elements by imposing $h(m,n)\coloneq[j_M(m),j_N(n)]$. To prove that it factors through $M\otimes N$ it suffices that $\overline{d_U}\comp h=0=\overline{d_L}\comp h$, the rest being trivial since $M\otimes N$ is the pullback of the kernels $K_{\overline{d_U}}$ and $K_{\overline{d_L}}$. This is done through the equalities
\begin{align*}
\overline{d_U} h(m,n)=\overline{d_U}([j_M(m),j_N(n))
=[\overline{d_U}j_M(m),\overline{d_U}j_N(n)]
=[\overline{d_U}j_M(m),0]=0
\end{align*}
and
\begin{align*}
\overline{d_L} h(m,n)&=\overline{d_L}([j_M(m),j_N(n)])
=[\overline{d_L}j_M(m),\overline{d_L}j_N(n)]
=[0,\overline{d_L}j_N(n)]=0.
\end{align*}
Thus we obtain $h\colon M\times N\to M\otimes N$. Let us prove that it is a Lie pairing. For 1.\ we have the chain of equalities
\begin{align*}
h([m,m'],n)&=[j_M([m,m']),j_N(n)]=[[j_M(m),j_M(m')],j_N(n)]\\
&=-[[j_M(m'),j_N(n)],j_M(m)]-[[j_N(n),j_M(m)],j_M(m')]\\
&=[j_M(m),[j_M(m'),j_N(n)]]-[j_M(m'),[j_M(m),j_N(n)]]\\
&=[j_M(m),j_N({}^{m'}n)]-[j_M(m'),j_N({}^{m}n)]
=h(m,{}^{m'}n)-h(m',{}^{m}n)
\end{align*}
and a similar one shows 2.; for 3.\ we have
\begin{align*}
h({}^{n}m,{}^{m'}n')&=[j_M({}^{n}m),j_N({}^{m'}n')]
=[[j_N(n),j_M(m)],[j_M(m'),j_N(n')]]\\
&=-[[j_M(m),j_N(n)],[j_M(m'),j_N(n')]]
=-[h(m,n),h(m',n')].
\end{align*}

By Proposition~\ref{prop:prop 1 in Ell91} we may take a universal Lie pairing~$\Tilde{h}$. This provides us with a unique morphism $\psi$ such that the triangle
\begin{equation}
\label{diag:triangle with Lie pairing}
\vcenter{
\xymatrix@!0@=3em{
& M\times N \ar[ld]_-{h} \ar[rd]^-{\Tilde{h}} \\
M\otimes N \ar@<-.5ex>@{.>}[rr]_-{\phi} && M\otimes_{\Lie} N \ar@<-.5ex>@{.>}[ll]_-{\psi}.
}
}
\end{equation}
commutes.  We next show that there is a unique $\phi$ such that $\phi\comp h=\Tilde{h}$. This then implies that $\phi$ and $\psi$ are each other's inverse, so that $M\otimes N\cong M\otimes_{\Lie}N$.

We use Lemma~\ref{lemma:the non-abelian tensor product of Lie algebras gives us a crossed square} which tells us that the non-abelian tensor product $M\otimes_{\Lie}N$ induces a crossed square of Lie algebras
\[
\xymatrix{
M\otimes_{\Lie}N \ar[r] \ar[d] & M \ar[d]\\
N \ar[r] & L
}
\]
in the sense of Definition~\ref{defi:explicit definition of crossed square of Lie algebras}. By Lemma~\ref{lemma:Lie crossed squares are the same as internal crossed squares in LieAlg} we know that Definition~\ref{defi:implicit defi of crossed square} in~$\LieAlg$ coincides with Definition~\ref{defi:explicit definition of crossed square of Lie algebras} and hence we can use the universal property of the tensor product ${M\otimes N}$---Proposition~\ref{prop:universal property of non-abelian tensor product}---which yields the needed unique morphism $\phi\colon {M\otimes N\to M\otimes_{\Lie}N}$ such that~\eqref{diag:triangle with Lie pairing} commutes.
\end{proof}

Hence from now on we may write $\rho_M$ and $\rho_N$ as $\pi^{M\otimes N}_M$ and $\pi^{M\otimes N}_N$, respectively.

\section{Internal crossed squares through the non-abelian tensor product}
\label{section:towards crossed squares through the non-abelian tensor product}

The aim of this section is to generalise the explicit description of crossed squares of groups (given in Definition~\ref{defi:explicit defi of crossed square of groups}) and Lie algebras (given in Definition~\ref{defi:explicit definition of crossed square of Lie algebras}) to the semi-abelian case, without passing through the double groupoid formalism. In order to do so, we use the construction of the non-abelian tensor product, first in the categories $\Grp$ and $\LieAlg$, and then in a general $\A$ (which is semi-abelian with \SH).
We call the object we obtain a \emph{weak crossed square}, and we prove that weak crossed squares are the same as crossed squares in $\Grp$ or $\LieAlg$. We then show that in the semi-abelian context, each double groupoid gives rise to a weak crossed square. 
The converse is still an open question: the aim would be to find suitable conditions on the surrounding category under which we have an equivalence. Under such conditions we have an explicit description of what is a crossed square.

The idea behind this internalisation is given by a bijection introduced in~\cite{BL87} (see their Definition~2.2 and following): the authors say that, given a pair of compatible group actions, to each crossed pairing $h\colon M\times N\to P$ corresponds a group homomorphism $h^*\colon M\otimes N\to P$ defined by $h^*(m\otimes n)=h(m,n)$. From now on we will write $h$ for both these maps, since there is no risk of confusion.

Using this hint as the basis of our reasoning we give the following definition.

\begin{definition}
\label{defi:weak crossed squares via the non-abelian tensor product}
Let $\A$ be a semi-abelian category that satisfies \SH. An \emph{(internal) weak crossed square in $\A$} is given by a commutative square
\[
\xymatrix{
P \ar@{.>}[rd]|-{\lambda} \ar[r]^-{p_M} \ar[d]_-{p_N} & M \ar[d]^-{\mu}\\
N \ar[r]_-{\nu} & L
}
\]
in $\A$, together with internal actions
\begin{align*}
\xi^L_M\colon L\flat M\to M && \xi^L_N\colon L\flat N\to N && \xi^L_P\colon L\flat P\to P
\end{align*}
and a morphism $h\colon M\otimes N\to P$ such that the following axioms hold:
\begin{enumerate}[{\rm (W.1)}]
\item the maps $p_M$ and $p_N$ are equivariant with respect to the $L$-actions: the squares
 \begin{align*}
 \xymatrix{
 L\flat P \ar[r]^-{\xi^L_P} \ar[d]_-{1_L\flat p_M} & P \ar[d]^-{p_M}\\
 L\flat M \ar[r]_-{\xi^L_M} & M
 }
 &&
 \xymatrix{
 L\flat P \ar[r]^-{\xi^L_P} \ar[d]_-{1_L\flat p_N} & P \ar[d]^-{p_N}\\
 L\flat N \ar[r]_-{\xi^L_N} & N
 }
 \end{align*}
 commute; furthermore, $(M\xrightarrow{\mu}L,\xi^L_M)$, $(N\xrightarrow{\nu}L,\xi^L_N)$ and ${(P\xrightarrow{\lambda}L,\xi^L_P)}$ are $L$-crossed modules;
 \item the diagram
 $
 \xymatrixcolsep{3pc}
 \xymatrixrowsep{2pc}
 \vcenter{\xymatrix{
 & M\otimes N \ar[rd]^-{\pi_M^{M\otimes N}} \ar[ld]_-{\pi_N^{M\otimes N}} \ar[d]_-{h}\\
 N & P \ar[r]_-{p_M} \ar[l]^-{p_N} & M
 }}
 $
 commutes;
 \item the diagram
 $
 \xymatrixcolsep{3pc}
 \xymatrixrowsep{2pc}
 \vcenter{\xymatrix{
 P\otimes N \ar[r]^-{p_M\otimes 1_N} \ar[rd]_-{\pi_P^{P\otimes N}} & M\otimes N \ar[d]_-{h} & M\otimes P \ar[l]_-{1_M\otimes p_N} \ar[ld]^-{\pi_P^{M\otimes P}}\\
 & P
 }}
 $
 commutes;
 \item the map $h$ is equivariant with respect to the action $\xi^L_{M\otimes N}\colon L\flat (M\otimes N)\to M\otimes N$ (induced as in Remark~\ref{rmk:definitions of diagonal action}); that is, the square 
 \[
 \xymatrix{
 L\flat(M\otimes N) \ar[d]_-{1_L\flat h} \ar[r]^-{\xi^L_{M\otimes N}} & M\otimes N \ar[d]^-{h}\\
 L\flat P \ar[r]_-{\xi^L_P} & P
 }
 \]
 commutes.
\end{enumerate}
A \emph{morphism of weak crossed squares}
\[
\xymatrixcolsep{1pc}
\xymatrixrowsep{1pc}
\xymatrix{
P \ar[rr]^-{p_M} \ar[dd]_-{p_N} && M \ar[dd]^-{\mu} &&&& P' \ar[rr]^-{p'_{M'}} \ar[dd]_-{p'_{N'}} && M' \ar[dd]^-{\mu'}\\
&&& \ar[rr]^-{\morph{p & f}{g & l}} && \\
N \ar[rr]_-{\nu} && L &&&& N' \ar[rr]_-{\nu'} && L'
}
\]
is given by a quadruple of morphisms 
\[
p\colon P\to P' \qquad f\colon M\to M'\qquad g\colon N\to N' \qquad l\colon L\to L'
\]
such that the obvious cube commutes and the $h$-maps are respected; that is, the square
\[
\xymatrix{
M\otimes N \ar[r]^-{f\otimes g} \ar[d]_-{h} & M'\otimes N' \ar[d]^-{h'}\\
P \ar[r]_-{p} & P'
}
\]
commutes as well.
\end{definition}

\begin{remark}
\label{rmk:defi of other actions from definition of crossed square}
From the three $L$-actions $\xi^L_M$, $\xi^L_N$ and $\xi^L_P$ we construct the actions $\xi^M_P$, $\xi^N_P$, $\xi^M_N$ and $\xi^N_M$ through the diagrams
\begin{align*}
\xymatrix{
M\flat P \ar[rd]_-{\xi^M_P} \ar[rr]^-{\mu\flat 1_P} && L\flat P \ar[ld]^-{\xi^L_P}\\
& P
}
&&
\xymatrix{
N\flat P \ar[rd]_-{\xi^N_P} \ar[rr]^-{\nu\flat 1_P} && L\flat P \ar[ld]^-{\xi^L_P}\\
& P
}
\\
\xymatrix{
M\flat N \ar[rd]_-{\xi^M_N} \ar[rr]^-{\mu\flat 1_N} &&  L\flat N \ar[ld]^-{\xi^L_N}\\
& N
}
&&
\xymatrix{
N\flat M \ar[rd]_-{\xi^N_M} \ar[rr]^-{\nu\flat 1_M} && L\flat M \ar[ld]^-{\xi^L_M}\\
& M
}
\end{align*}
Condition \W1 implies that also $(P\xrightarrow{p_M}M,\xi^M_P)$ and $(P\xrightarrow{p_N}N,\xi^N_P)$ are crossed modules. This is an application of the following lemma.
\end{remark}

\begin{lemma}
Let $\A$ be a semi-abelian category with \SH. Consider a triangle
\[
\xymatrix@!0@!=1em{
P \ar[rr]^-{p} \ar[rd]_-{\lambda} && M\ar[dl]^-{\mu} \\
& L
}
\]
with internal crossed module structures $(M\xrightarrow{\mu}L,\xi^L_M)$ and $(P\xrightarrow{\lambda}L,\xi^L_P)$, and the induced action $\xi^M_P\coloneq\xi^L_P\comp (\mu\flat 1_P)$. If $p$ is equivariant with respect to the $L$-actions, i.e., the square
\[
\xymatrix{
L\flat P \ar[r]^-{\xi^L_P} \ar[d]_-{1_L\flat p} & P \ar[d]^-{p}\\
L\flat M \ar[r]_-{\xi^L_M} & M
}
\]
commutes, then also $(P\xrightarrow{p}M,\xi^M_P)$ is an internal crossed module.
\end{lemma}
\begin{proof}
We need to show the commutativity of the squares
\[
\xymatrix{
P\flat P \ar[r]^-{\chi_P} \ar[d]_-{p\flat 1_P} & P \ar@{=}[d]\\
M\flat P \ar[r]_-{\xi^M_P} & P
}
\qquad\qquad
\xymatrix{M\flat P \ar[r]^-{\xi^M_P} \ar[d]_-{1_M\flat p} & P \ar[d]^-{p}\\
M\flat M \ar[r]_-{\chi_M} & M.}
\]
For the left one we have the chain of equalities
\[
\xi^M_P\comp(p\flat 1_P)=\xi^L_P\comp (\mu\flat 1_P)\comp(p\flat 1_P)=\xi^L_P\comp (\lambda\flat 1_P)=\chi_P,
\]
and we have
\begin{align*}
p\comp\xi^M_P=p\comp\xi^L_P\comp (\mu\flat 1_P)=\xi^L_M \comp(1_L\flat p)\comp (\mu\flat 1_P)
=\xi^L_M \comp(\mu\flat 1_M)\comp(1_L\flat p)=\chi_M\comp(1_L\flat p)
\end{align*}
for the right one.
\end{proof}

\begin{proposition}
\label{prop:explicit definition of crossed squares in grp coincides with the explicit definition of crossed square in A}
If $\A=\Grp$, then weak crossed squares are the same as internal crossed squares, that is the group version of Definition~\ref{defi:weak crossed squares via the non-abelian tensor product} is equivalent to Definition~\ref{defi:explicit defi of crossed square of groups}.
\end{proposition}
\begin{proof}
As explained in~\cite{BL87}, given a crossed pairing $h\colon M\times N\to P$ (see Remark~\ref{rmk:h of a crossed square is a crossed pairing}), we can decompose it as 
\[
\xymatrix@!0@!=1em{
M\times N \ar[rd]_-{h} \ar[rr]^-{-\otimes -} && M\otimes N \ar[ld]^-{h^*}\\
& P
}
\]
where the horizontal map, which sends $(m,n)$ to $m\otimes n$, is called the \emph{universal crossed pairing}, whereas $h^*$ is a morphism of groups. Vice versa, we can associate a crossed pairing $h^*\comp(-\otimes -)$ to every morphism of groups $h^*\colon M\otimes N\to P$. This means that giving a crossed pairing amounts to giving a morphism out of the non-abelian tensor product. For the sake of simplicity, we are going to denote both of them as $h$.

Notice that \W1 is precisely the internal reformulation of \X1, which makes them equivalent. Let us prove that \X2 $\iff$ \W2. The only non-trivial step is given by the explicit description
\[
\pi_M^{M\otimes N}(m\otimes n)=m{}^{n}m^{-1},\qquad\qquad
\pi_N^{M\otimes N}(m\otimes n)={}^{m}nn^{-1}
\]
for the projection maps: for further details see Proposition~2.3.b in~\cite{BL87}. Using these equations together with $p_M(h(m\otimes n))=p_M(h(m,n))$ and $
p_N(h(m\otimes n))=p_N(h(m,n))$ we obtain the desired equivalence.

Similarly, in order to show \X3 $\iff$ \W3, we use the equations
\begin{align*}
\begin{cases}
\pi_P^{P\otimes N}(p\otimes n)=p{}^{n}p^{-1},\\ \pi_P^{M\otimes P}(m\otimes p)={}^{m}pp^{-1},
\end{cases}
&&
\begin{cases}
h(p_M(p)\otimes n)=h(p_M(p),n),\\
h(m\otimes p_N(p))=h(m,p_N(p)).
\end{cases}
\end{align*}
We have already explained that, whenever \X4 holds, \X0 is equivalent to the requirement that $h\colon M\times N\to P$ is a crossed pairing, and that this is in turn equivalent to having a morphism $h\colon M\otimes N\to P$.

Finally, to show that \X4 $\iff$ \W4, we first take the action $\xi^L_{M\otimes N}$ as defined in Remark~\ref{rmk:definitions of diagonal action}: in the particular case of groups it can be described through the equation 
\[
{}^l(m\otimes n)=({}^lm)\otimes({}^ln).
\]
(For more details about this action see Proposition~2.3.a in~\cite{BL87}). Then the equations
\[
{}^lh(m\otimes n)={}^lh(m,n),
\qquad
h({}^lm\otimes{}^ln)=h({}^lm,{}^ln)
\]
prove our claim.
\end{proof}

\begin{remark}
\label{rmk:the morphism of crossed squares to the tensor product one and the map induced from the crossed pairing are the same}
Consider a crossed square of groups as in Definition~\ref{defi:explicit defi of crossed square of groups}: according to Proposition~\ref{prop:universal property of non-abelian tensor product} we have a unique morphism $\phi\colon M\otimes N\to P$ such that $\morph{\phi & 1_M}{1_N & 1_L}$ is a morphism of crossed squares. In particular this map $\phi$ is the same as the map ${h\colon M\otimes N\to P}$ induced by the crossed pairing $h\colon M\times N\to P$. To see this, it suffices to show that $\morph{h & 1_M}{1_N & 1_L}$ is again a morphism of crossed squares: following the description of morphisms given in Definition~\ref{defi:weak crossed squares via the non-abelian tensor product}, this amounts to proving that $h$ makes the outer cube in Figure~\ref{Fig:universal property of non-abelian tensor product} commute as well as the diagram
\[
\xymatrix{
M\otimes N \ar@{=}[r] \ar@{=}[d] & M\otimes N \ar[d]^-{h}\\
M\otimes N \ar[r]_-{h} & P.
}
\]
The latter is trivial, and the former is given by condition \W2.
\end{remark}

Using the same reasoning as in the last remark we can deduce that, if there is a way to show the equivalence between the notion of weak crossed square and the one of internal crossed square, then the morphism $h\colon M\otimes N\to P$ has to be the one given by Proposition~\ref{prop:universal property of non-abelian tensor product}.

\begin{proposition}
\label{prop:explicit definition of crossed squares in LieAlg coincides with the explicit definition of crossed square in A}
In $\A=\LieAlg$, weak crossed squares are the same as internal crossed squares, so that the Lie algebra version of Definition~\ref{defi:weak crossed squares via the non-abelian tensor product} coincides with Definition~\ref{defi:explicit definition of crossed square of Lie algebras}.
\end{proposition}
\begin{proof}
Let us compare condition \X0--\X4 as in Definition~\ref{defi:explicit definition of crossed square of Lie algebras} with condition \W1--\W4 as in Definition~\ref{defi:weak crossed squares via the non-abelian tensor product}.

As follows from Proposition~\ref{prop:prop 1 in Ell91}, having a function $h\colon M\times N\to P$ such that \X0 holds (that is a Lie pairing) is the same as having a morphism $h^*\colon M\otimes N\to P$ (from now on denoted again with $h$).

Notice that \W1 is precisely the internal reformulation of \X1, and hence they are clearly equivalent. The equivalence \X2 $\iff$ \W2 is given by the equivalence between the systems
\begin{align*}
\begin{cases}
\pi_M^{M\otimes N}(m\otimes n)=p_M(h(m\otimes n)),\\ \pi_N^{M\otimes N}(m\otimes n)=p_N(h(m\otimes n)), 
\end{cases}
&&
\begin{cases}
-{}^{n}m=p_M(h(m,n)),\\ {}^{m}n=p_N(h(m,n)), 
\end{cases}
\end{align*}
which in turn is obtained via the explicit description of the maps $\pi_M^{M\otimes N}$ and $\pi_M^{M\otimes N}$ in the Lie algebra case.
Similarly, in order to show \X3 $\iff$ \W3, we use the equivalence between the systems
\begin{align*}
\begin{cases}
\pi_P^{P\otimes N}(p\otimes n)=h((p_N\otimes 1_N)(p\otimes n)),\\ \pi_P^{M\otimes P}(m\otimes p)=h((1_M\otimes p_M)(m\otimes p)),
\end{cases}
&&
\begin{cases}
h(p_M(p),n)=-{}^{n}p,\\
h(m,p_N(p))={}^{m}p.
\end{cases}
\end{align*}
Finally, to show that \X4 $\iff$ \W4, we first consider the action $\xi^L_{M\otimes N}$ as defined in Remark~\ref{rmk:definitions of diagonal action}: in the particular case of Lie algebras it can be described through the equation 
\[
{}^l(m\otimes n)=({}^lm)\otimes n+m\otimes({}^ln).
\]
Then we use the equivalent equalities
\begin{align*}
& & h({}^{l}(m\otimes n))&={}^{l}(h(m\otimes n))\\
& \Leftrightarrow & h({}^lm\otimes n+m\otimes{}^ln)&={}^{l}(h(m\otimes n))\\
& \Leftrightarrow & h({}^lm\otimes n)+h(m\otimes{}^ln)&={}^{l}(h(m\otimes n))\\
& \Leftrightarrow & h({}^lm,n)+h(m,{}^ln)&={}^{l}(h(m,n))
\end{align*}
to finish the proof.
\end{proof}

We return to the context of a semi-abelian category $\A$ that satisfies \SH.

\begin{proposition}
\label{prop:implicit defi implies explicit defi}
An internal crossed square is always a weak crossed square, that is Definition~\ref{defi:implicit defi of crossed square} implies Definition~\ref{defi:weak crossed squares via the non-abelian tensor product}.
\end{proposition}
\begin{proof}
Consider a normalisation of an internal double groupoid as in Figure~\ref{FigNormalisation}. 
\begin{figure}

\resizebox{.35\textwidth}{!}
{\xymatrix@!0@=6em{
P \ar@{{ |>}->}[r]^-{k_{d_T}} \ar@{.>}[rd]|-{k} \ar@{{ |>}->}[d]_-{k_{d_W}} & P\rtimes M \ar@<1ex>[r]^-{d_T} \ar@<-1ex>[r]_-{c_T} \ar@{{ |>}->}[d]^-{k_{d_L}} & M \ar[l]|-{e_T} \ar@{{ |>}->}[d]^-{k_{d_R}}\\
P\rtimes N \ar@<1ex>[d]^-{d_W} \ar@<-1ex>[d]_-{c_W} \ar@{{ |>}->}[r]_-{k_{d_U}} & Q_P \ar@<1ex>[d]^-{d_L} \ar@<-1ex>[d]_-{c_L} \ar@<1ex>[r]^-{d_U} \ar@<-1ex>[r]_-{c_U} \ar@<3pt>@{.>}[dr]^-{d} \ar@<-3pt>@{.>}[dr]_-{c} & M\rtimes L \ar[l]|-{e_U} \ar@<1ex>[d]^-{d_R} \ar@<-1ex>[d]_-{c_R}\\
N \ar[u]|-{e_W} \ar@{{ |>}->}[r]_-{k_{d_D}} & N\rtimes L \ar[u]|-{e_L} \ar@<1ex>[r]^-{d_D} \ar@<-1ex>[r]_-{c_D} & L \ar[u]|-{e_R} \ar[l]|-{e_D} \ar@{.>}[lu]|-{e}
}}
\caption{An internal double groupoid and its normalisation}\label{FigNormalisation}
\end{figure}
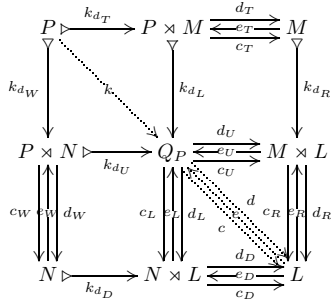
Let us start by fixing the basic ingredients. We define the maps $p_M\coloneq c_T\comp k_{d_T}$, $p_N\coloneq c_W\comp k_{d_W}$ and $\lambda\coloneq c\comp k$. The actions $\xi^L_M$ and $\xi^L_N$ are already given, whereas $\xi^L_P$ and $\xi^L_{M\otimes N}$ are constructed as in~\eqref{diag:definition of diagonal groupoid} and $h\colon M\otimes N\to P$ is given by Proposition~\ref{prop:universal property of non-abelian tensor product}. Now we are ready to show the properties of these objects.

For \W1, we already know by hypothesis that $(M\xrightarrow{\mu}L,\xi^L_M)$ and $(N\xrightarrow{\nu}L,\xi^L_N)$ are crossed modules. The fact that also $(P\xrightarrow{\lambda}L,\xi^L_P)$ is so, is given by Proposition~\ref{prop:the diagonal is a crossed module}. It remains to be shown that $p_M\colon P\to M$ is equivariant with respect to these actions; then for $p_N$ the reasoning is entirely similar. Consider the diagram
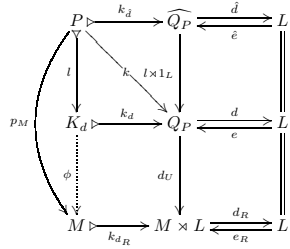
\[
\resizebox{.3\textwidth}{!}
{\xymatrix@!0@=6em{
P \ar@/_2pc/[dd]_-{p_M} \ar@{{ |>}->}[d]_-{l} \ar[rd]|-{k} \ar@{{ |>}->}[r]^-{k_{\Hat{d}}} & \widehat{Q_P} \ar[d]_-{l\rtimes 1_L} \ar@<.5ex>[r]^-{\Hat{d}} & L \ar@{=}[d] \ar@<.5ex>[l]^-{\Hat{e}}\\
K_d \ar@{.>}[d]_-{\phi} \ar@{{ |>}->}[r]^-{k_d} & Q_P \ar[d]_-{d_U} \ar@<.5ex>[r]^-{d} & L \ar@{=}[d] \ar@<.5ex>[l]^-{e}\\
M \ar@{{ |>}->}[r]_-{k_{d_R}} & M\rtimes L \ar@<.5ex>[r]^-{d_R} & L \ar@<.5ex>[l]^-{e_R}
}}
\]
where the two top squares are the ones defining the action $\xi^L_P$, whereas the dotted map is induced by the fact that $M$ is the kernel of $d_R$. In order to show that $(p_M,1_L)$ is a morphism of split extensions from the top row to the bottom one (and hence an equivariant map), it suffices that $p_M=\phi\comp l$, since each square commutes: this is done using the chain of equalities $k_{d_R}\comp \phi\comp l=d_U\comp k_d\comp l=d_U\comp k=k_{d_R}\comp p_M$ and the fact that $k_{d_R}$ is a monomorphism.

Condition \W2 is already given by definition of the map $h$. In order to show \W3 it suffices that 
\[
\resizebox{.5\textwidth}{!}
{\xymatrix@!0@=2.5em{
M\otimes P \ar[rr]^-{\pi^{M\otimes P}_M} \ar[dd]_-{\pi^{M\otimes P}_P} && M \ar[dd]^-{\mu} &&&&& P \ar[rr]^-{p_M} \ar[dd]_-{p_N} && M \ar[dd]^-{\mu}\\
&&& \ar@<.5ex>[rrr]^-{\morph{\pi^{M\otimes P}_P & 1_M}{p_N& 1_L}} \ar@<-.5ex>[rrr]_-{\morph{h\circ(1_M\otimes p_N) & 1_M}{p_N &1_L}} &&& \\
P \ar[rr]_-{\lambda} && L &&&&& N \ar[rr]_-{\nu} && L
}}
\]
are both morphisms of crossed squares, so that the claim follows from Proposition~\ref{prop:universal property of non-abelian tensor product}: the universal property of $M\otimes P$ (and similarly for $P\otimes N$). The map $\pi^{M\otimes P}_P$ clearly satisfies the universal property depicted in Proposition~\ref{prop:universal property of non-abelian tensor product} and therefore it induces the morphism of crossed squares on the top. The second one is obtained as the composition
\[
\morph{h\circ(1_M\otimes p_N)& 1_M}{p_N& 1_L}=\morph{h& 1_M}{1_P &1_L}\comp \morph{1_M\otimes p_N & 1_M}{p_N & 1_L}.
\]
The first one is a morphism of crossed squares by definition of $1\otimes p_N$, whereas the second one is so by definition of $h$ (and by Remark~\ref{rmk:the morphism of crossed squares to the tensor product one and the map induced from the crossed pairing are the same}).

It remains to be shown that \W4 holds and to do so, consider Figure~\ref{Fig:(iv')}.
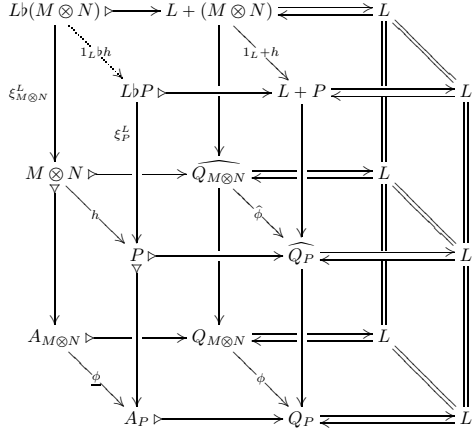
\begin{figure}

\resizebox{.5\textwidth}{!}{
\xymatrix@!0@=4.5em{
L\flat(M\otimes N) 
\ar@{{ |>}->}[rr]
\ar[dd]_-{\xi^L_{M\otimes N}}
\ar@{.>}[rd]|-{1_L\flat h}
&&
L+(M\otimes N)
\ar@<2pt>[rr]
\ar[dd]|-{\hole}
\ar[rd]|-{1_L+h}
&&
L
\ar@<2pt>[ll]
\ar@{=}[dd]|-{\hole}
\ar@{=}[rd]
\\
&
L\flat P
\ar@{{ |>}->}[rr]
\ar[dd]_(0.25){\xi^L_P}
&&
L+P
\ar@<2pt>[rr]
\ar[dd]
&&
L
\ar@<2pt>[ll]
\ar@{=}[dd]
\\
M\otimes N
\ar@{{ |>}->}[dd]
\ar@{{ |>}->}[rr]|(.50){\hole}
\ar[rd]|-{h}
&&
\widehat{Q_{M\otimes N}}
\ar[dd]|(.50){\hole}
\ar@<2pt>[rr]|(.50){\hole}
\ar[rd]|-{\widehat{\phi}}
&&
L
\ar@{=}[dd]|(.50){\hole}
\ar@<2pt>[ll]|(.50){\hole}
\ar@{=}[rd]
\\
&
P
\ar@{{ |>}->}[dd]
\ar@{{ |>}->}[rr]
&&
\widehat{Q_P}
\ar[dd]
\ar@<2pt>[rr]
&&
L
\ar@{=}[dd]
\ar@<2pt>[ll]
\\
A_{M\otimes N}
\ar@{{ |>}->}[rr]|(.50){\hole}
\ar[rd]|-{\underline{\phi}}
&&
Q_{M\otimes N}
\ar[rd]|-{\phi}
\ar@<2pt>[rr]|(.50){\hole}
&&
L
\ar@<2pt>[ll]|(.50){\hole}
\ar@{=}[rd]
\\
&
A_P
\ar@{{ |>}->}[rr]
&&
Q_P
\ar@<2pt>[rr]
&&
L
\ar@<2pt>[ll]
}}

\caption{Checking condition \W4}\label{Fig:(iv')}
\end{figure}
We want to prove that the top square in the left face commutes. Notice that by definition of $\phi$ and $\underline{\phi}$ we already know that the squares on the bottom face commute, and similarly, by definition of $h$ the bottom square on the left face commutes. The two lower cubes then commute by construction of $\widehat{Q_{M\otimes N}}$, $\widehat{Q_P}$ and $\Hat{\phi}$ (see Lemma~2.6 in~\cite{CGVdL15a}). This means that $(h,1_L)$ is a morphism between the lifted points and therefore $h$ is equivariant.
\end{proof}

\subsection{When is a weak crossed square a crossed square?}

\begin{figure}

\resizebox{.5 \textwidth}{!}{
\xymatrix@!0@=4.5em{
M\otimes N 
\ar@{{ |>}->}[rr]
\ar@{{ |>}->}[dd]
\ar@{-{ >>}}[rd]|-{h}
&&
(M\otimes N)\rtimes M
\ar@<3pt>[rr]
\ar@<-3pt>[rr]
\ar@{{ |>}->}[dd]|-{\hole}
\ar@{-{ >>}}[rd]|-{h\rtimes 1_M}
&&
M
\ar[ll]
\ar@{{ |>}->}[dd]|-{\hole}
\ar@{=}[rd]
\\
&
P
\ar@{{ |>}->}[rr]
\ar@{{ |>}->}[dd]
&&
P\rtimes M
\ar@<3pt>[rr]
\ar@<-3pt>[rr]
\ar@{{ |>}.>}[dd]^(.3){\alpha}
&&
M
\ar[ll]
\ar@{{ |>}->}[dd]
\\
(M\otimes N)\rtimes N
\ar@<3pt>[dd]
\ar@<-3pt>[dd]
\ar@{{ |>}->}[rr]|(.50){\hole}
\ar@{-{ >>}}[rd]|-{h\rtimes 1_N}
&&
Q_{M\otimes N}
\ar@<3pt>[dd]|(.50){\hole}
\ar@<-3pt>[dd]|(.50){\hole}
\ar@<3pt>[rr]|(.50){\hole}
\ar@<-3pt>[rr]|(.50){\hole}
\ar@{.>}[rd]|-{\Tilde{h}}
&&
M\rtimes L
\ar@<3pt>[dd]|(.50){\hole}
\ar@<-3pt>[dd]|(.50){\hole}
\ar[ll]|(.50){\hole}
\ar@{=}[rd]
\\
&
P\rtimes N
\ar@<3pt>[dd]
\ar@<-3pt>[dd]
\ar@{{ |>}.>}[rr]_(.3){\beta}
&&
Q'
\ar@{.>}@<3pt>[dd]
\ar@{.>}@<-3pt>[dd]
\ar@{.>}@<3pt>[rr]
\ar@{.>}@<-3pt>[rr]
&&
M\rtimes L
\ar@<3pt>[dd]
\ar@<-3pt>[dd]
\ar@{.>}[ll]
\\
N
\ar[uu]
\ar@{{ |>}->}[rr]|(.50){\hole}
\ar@{=}[rd]
&&
N\rtimes L
\ar[uu]|(.50){\hole}
\ar@{=}[rd]
\ar@<3pt>[rr]|(.50){\hole}
\ar@<-3pt>[rr]|(.50){\hole}
&&
L
\ar[uu]|(.50){\hole}
\ar[ll]|(.50){\hole}
\ar@{=}[rd]
\\
&
N
\ar[uu]
\ar@{{ |>}->}[rr]
&&
N\rtimes L
\ar@{.>}[uu]
\ar@<3pt>[rr]
\ar@<-3pt>[rr]
&&
L
\ar[uu]
\ar[ll]
}}

\caption{The crossed square induced by a regular epic $h$}\label{Fig:CrossedSquareByh}
\end{figure}
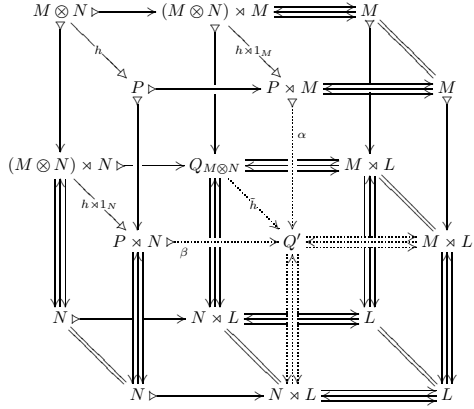

It remains an open question whether the converse of Proposition~\ref{prop:implicit defi implies explicit defi} holds; a stronger condition on the base category $\A$ might be necessary for this to be the case. 

We have a partially positive answer in the situation where $h$ happens to be a regular epimorphism: such a weak crossed square is always a crossed square, as soon as in the induced diagram of Figure~\ref{Fig:CrossedSquareByh}, the kernel of~$h$ is normal in $Q_{M\otimes N}$. 

Note, however, that examples of crossed squares exist where the induced $h$ is \emph{not} a regular epimorphism---see for instance Subsection~\ref{SubsecNormalSubobjects}. For this reason, what follows here can only ever be a partial answer to the question. 

As it turns out, a double groupoid as in Figure~\ref{FigNormalisation} gives rise to a regular epimorphic~$h$ (whose kernel is necessarily normal in $Q_{M\otimes N}$) if and only if the morphisms $e_L$ and $e_U$ are jointly extremal-epimorphic. Indeed, the latter condition holds if and only if the morphism $\Tilde{h}$ in Figure~\ref{Fig:CrossedSquareByh} is a regular epimorphism. We may then use the idea contained in the following remark.

\begin{remark}
\label{rmk:if h is regular epimorphism then tilde(h)=coker(ker(h))}
Suppose for the moment that the front face in Figure~\ref{Fig:CrossedSquareByh} is already an internal crossed square. Then both squares in the diagram
\begin{equation}
\label{diag:the outer rectangle is a pushout}
\vcenter{
\xymatrix{
M\otimes N \pullbackcorner \ar[r] \ar@{-{ >>}}[d]_-{h} & (M\otimes N)\rtimes M \pullbackcorner \ar@{-{ >>}}[d]|-{h\rtimes 1_M} \ar[r] & Q_{M\otimes N} \ar@{-{ >>}}[d]^-{\Tilde{h}}\\
P \ar[r] & P\rtimes M \ar[r]_-{\alpha} & Q'
}
}
\end{equation}
are pullbacks (by item 1.\ of Lemma~4.2.4 in~\cite{BB04}) and hence the outer rectangle is so. By item 2.\ of Lemma~4.2.4 in~\cite{BB04}, this implies that $K_h\cong K_{\Tilde{h}}$, but since~$\Tilde{h}$ is a regular epimorphism if and only if so is $h$ (by applying the Short $5$-Lemma twice), it is the cokernel of its kernel: this means that~$Q'$ can be described as the cokernel of the inclusion of $K_h$ into $Q_{M\otimes N}$. Furthermore, this inclusion is normal.
\end{remark}

Conversely, when $h$ is a regular epimorphism and the kernel of~$h$ is normal in $Q_{M\otimes N}$, we can construct the object $Q'$ and the dotted arrows in Figure~\ref{Fig:CrossedSquareByh} so that the double reflexive graph in the front face is an internal double groupoid:

\begin{theorem}
\label{prop:other implication in the regular epimorphism case}
In a semi-abelian category that satisfies \SH, a weak crossed square where $h$ is a regular epimorphism is also an internal crossed square---that is, Definition~\ref{defi:weak crossed squares via the non-abelian tensor product} implies Definition~\ref{defi:implicit defi of crossed square} in that case---as soon as in the induced diagram of Figure~\ref{Fig:CrossedSquareByh}, the kernel of $h$ is normal in $Q_{M\otimes N}$.
\end{theorem}
\begin{proof}
By using the idea in the previous remark we define $Q'$ as the cokernel of $\gamma\comp k_h$, where $\gamma$ is the composition depicted in the first row of~\eqref{diag:the outer rectangle is a pushout}. In particular we obtain that 
\begin{equation}
\label{diag:Q' comes from a pushout}
\vcenter{
\xymatrix{
M\otimes N \ar[r]^-{\gamma} \ar@{-{ >>}}[d]_-{h} & Q_{M\otimes N} \ar@{-{ >>}}[d]^-{\Tilde{h}}\\
P \ar[r]_-{\gamma'} & Q' \pushoutcorner
}
}
\end{equation}
is a pushout. Since $Q'$ is the cokernel of $\gamma\comp k_h$, from $d_U\comp\gamma=0=d_L\comp\gamma$ we find unique morphisms $d'_U\colon Q'\to M\rtimes L$, $d'_L\colon Q'\to N\rtimes L$ such that $d'_U\comp \Tilde{h}=d_U$ and $d'_L\comp \Tilde{h}=d_L$. Similarly, by using the universal property of the pushout~\eqref{diag:Q' comes from a pushout} we obtain unique morphisms $c'_U\colon Q'\to M\rtimes L$, $c'_L\colon Q'\to N\rtimes L$ such that $c'_U\comp \Tilde{h}=c_U$ and $c'_L\comp \Tilde{h}=c_L$. Then we define $e'_U\coloneq\Tilde{h}\comp e_U$ and $e'_L\coloneq\Tilde{h}\comp e_L$. With these data we already have that $(Q',M\rtimes L,d'_U,c'_U,e'_U)$ and $(Q',N\rtimes L,d'_L,c'_L,e'_L)$ are reflexive graphs. Since they are quotients of groupoids, they are groupoids as well. In particular, the square of groupoids involving them is a double groupoid: this can be shown by proving the commutativity of each of the nine squares by using the fact that $\Tilde{h}$ is a regular epimorphism.

We still need to construct morphisms $\alpha\colon P\rtimes M\to Q'$ and $\beta\colon P\rtimes N\to Q' $ making Figure~\ref{Fig:CrossedSquareByh} commute, and show that $\alpha=k_{d'_L}$ and $\beta=k_{d'_U}$. We are going to construct $\alpha$ only, since a symmetric strategy works for $\beta$.
Let us first of all notice that the square
\begin{equation}
\label{diag:a commutative square for the construction of alpha}
\vcenter{
\xymatrixcolsep{3pc}
\xymatrix{
M+(M\otimes N) \ar[r]^-{\binom{e_U\comp k^L_M}{\gamma}} \ar[d]_-{1+h} & Q_{M\otimes N} \ar[d]^-{\Tilde{h}}\\
M+P \ar[r]_-{\binom{e'_U\comp k^L_M}{\gamma'}} & Q'
}
}
\end{equation}
is commutative by definition of $e_U'$ and the commutativity of~\eqref{diag:Q' comes from a pushout}. Also the triangle 
\begin{equation}
\label{diag:a commutative triangle for the construction of alpha}
\vcenter{
\xymatrix{
M+(M\otimes N) \ar[rd]_-{\binom{e_U\comp k^L_M}{\gamma}} \ar[rr]^-{\sigma_{\xi^M_{M\otimes N}}} && (M\otimes N)\rtimes M \ar[ld]^-{k_{d_L}}\\
& Q_{M\otimes N}
}
}
\end{equation}
commutes, since
\[
\vcenter{
\xymatrix{
M \ar[d]_-{k^L_M} \ar[r]^-{e_T} & (M\otimes N)\rtimes M \ar[d]^-{k_{d_L}}\\
M\rtimes L \ar[r]_-{e_U} & Q_{M\otimes N}
}
}
\quad\text{and}\quad
\vcenter{
\xymatrix{
M\otimes N \ar[rd]_-{\gamma} \ar[rr]^-{k_{d_T}} && (M\otimes N)\rtimes M \ar[dl]^-{k_{d_L}}\\
& Q_{M\otimes N}
}
}
\]
do. Now we can use the definition of the semidirect product $P\rtimes M$ as a coequaliser to obtain the dotted arrow $\alpha$ in the commutative diagram of solid arrows
\[
\resizebox{.8\textwidth}{!}
{\xymatrix@!0@C=8em@R=4em{
M\flat(M\otimes N)
 \ar@{-{ >>}}[dd]_{1_M\flat h}
 \ar@<.5ex>[rr]^-{k_{M,M\otimes N}}
 \ar@<-.5ex>[rr]_-{i_{M\otimes N}\circ\xi^M_{M\otimes N}}
&&
M+(M\otimes N)
 \ar@{-{ >>}}[dd]_-{1_M+h}
 \ar[rr]^-{\sigma_{\xi^M_{M\otimes N}}}
 \ar[rd]_-{\binom{e_U\circ k^L_M}{\gamma}}
&&
(M\otimes N)\rtimes M
 \ar[dd]^-{h\rtimes 1_M}
 \ar[ld]^-{k_{d_L}}
\\
&&&
Q_{M\otimes N}
\ar[dd]^(.3){\Tilde{h}}
\\
M\flat P
 \ar@<.5ex>[rr]^-{k_{M,P}}
 \ar@<-.5ex>[rr]_-{i_P\circ\xi^M_P}
&&
M+P
 \ar[rr]^(.3){\sigma_{\xi^M_P}}|-{\hole}
 \ar[rd]_-{\binom{e'_U\circ k^L_M}{\gamma'}}
&&
P\rtimes M.
 \ar@{.>}[ld]^-{\alpha}
\\
&&&
Q'
}}
\]
In particular we need to show that $\binom{e'_U\circ k^L_M}{\gamma'}$ coequalises $k_{M,P}$ and $i_P\comp \xi^M_P$: this is done by precomposing with the regular epimorphism $1_M\flat h$ and by using the commutativity of \eqref{diag:a commutative square for the construction of alpha} and \eqref{diag:a commutative triangle for the construction of alpha}.
In a similar way we build $\beta\colon P\rtimes N\to Q'$. Let us now show that every square in Figure~\ref{Fig:CrossedSquareByh} involving $\alpha$ and $\beta$ commutes. Indeed, we already know that the square
 \[
 \xymatrix{
 (M\otimes N)\rtimes M \ar[r]^-{k_{d_L}} \ar[d]_-{h\rtimes 1} & Q_{M\otimes N} \ar[d]^-{\Tilde{h}}\\
 P\rtimes M \ar[r]_-{\alpha} & Q'
 }
 \]
 commutes by construction and similarly for the one involving $\beta$; the square 
 \[
 \xymatrix{
 P \ar[rd]|-{\gamma'} \ar[r]^-{k^M_P} \ar[d]_-{k^N_P} & P\rtimes M \ar[d]^-{\alpha}\\
 P\rtimes N \ar[r]_-{\beta} & Q'
 }
 \]
 commutes by construction of $\alpha$ and $\beta$; finally we need to show that the two right-pointing squares and the left-pointing square in
 \[
 \xymatrixrowsep{3pc}
 \xymatrixcolsep{3pc}
 \xymatrix{
 P\rtimes M \ar[d]_-{\alpha} \ar@<1ex>[r]^-{d^M_P} \ar@<-1ex>[r]_-{c^M_P} & M \ar[l]|-{e^M_P} \ar[d]^-{k^L_M}\\
 Q' \ar@<1ex>[r]^-{d'_U} \ar@<-1ex>[r]_-{c'_U} & M\rtimes L \ar[l]|-{e'_U}
 }
 \]
 commute. For the left-pointing one we have the chain of equalities
 \begin{align*}
 \alpha\comp e^M_P=\alpha\comp\sigma_{\xi^M_P}\comp i_M=\binom{e'_U\circ k^L_M}{\gamma'}\comp i_M=e'_U\comp k^L_M,
 \end{align*}
 whereas for the right-pointing ones we need to compose with the regular epimorphism~$\sigma_{\xi^M_P}$ to obtain
 \begin{align*}
 d'_U\comp\alpha\comp\sigma_{\xi^M_P}&=\binom{d'_U\circ e'_U\comp k^L_M}{d'_U\circ\gamma'}=\binom{k^L_M}{0}=k^L_M\binom{1_M}{0}
 =k^L_M\comp d^M_P\comp\sigma_{\xi^M_P},\\
 c'_U\comp\alpha\comp\sigma_{\xi^M_P}&=\binom{c'_U\circ e'_U\circ k^L_M}{c'_U\circ\gamma'}=\binom{k^L_M}{k^L_M\circ p_M}=k^L_M\comp \binom{1_M}{p_M}
 =k^L_M\comp c^M_P\comp\sigma_{\xi^M_P}.
 \end{align*}
 Finally we can repeat this argument for the corresponding squares involving $\beta$.
 
It remains to be shown that $\alpha=k_{d'_L}$ (and similarly that $\beta=k_{d'_U}$): to do this, we first show that $d'_L$ is the cokernel of $\alpha$ and then that $\alpha$ is a normal monomorphism, which implies the claim. The first step is easily done by checking the universal property of the cokernel through the diagram
\[
\xymatrix{
(M\otimes N)\rtimes M \ar[r]^-{k_{d_L}} \ar[d]_-{h\rtimes 1_M} & Q_{M\otimes N} \ar[r]^-{d_L} \ar[d]_-{\Tilde{h}} & N\rtimes L \ar@{=}[d]\\
P\rtimes M \ar[r]_-{\alpha} & Q' \ar[r]_-{d'_L} & N\rtimes L
}
\]
and the universal property of the cokernel $d_L$.

If the kernel of $h$ is normal in~$Q_{M\otimes N}$, then~\eqref{diag:Q' comes from a pushout}---which is precisely the outer rectangle in~\eqref{diag:the outer rectangle is a pushout}---is a pullback by~\cite[Lemma~4.2.4]{BB04}. For the same reason, also the left hand side square in~\eqref{diag:the outer rectangle is a pushout} is a pullback. Since $h\rtimes 1_M$ is a regular epimorphism, Proposition~4.1.4 in~\cite{BB04} now implies that the right hand side square in~\eqref{diag:the outer rectangle is a pushout} is a pullback. Since $\A$ is protomodular, pullbacks reflect monomorphisms; since $k_{d_L}$ is a monomorphism, so is~$\alpha$. Furthermore, $\alpha$ is normal as a direct image of the normal monomorphism $k_{d_L}$, which implies our claim that $\alpha$ is the kernel of $d'_L$. 
\end{proof}

\section*{Acknowledgements}
We would like to thank Alan Cigoli, Marino Gran, Manfred Hartl, Sandra Mantovani, Giuseppe Metere, Andrea Montoli and Cyrille S.\ Simeu for many fruitful discussions on the subject of the article. Special thanks to Diana Rodelo for a precious remark concerning the proof of Theorem~\ref{prop:other implication in the regular epimorphism case}, and to the referee for careful reading and providing us with numerous suggestions which helped improve the paper. The first author thanks the Université catholique de Louvain for its kind hospitality during his stays in Louvain-la-Neuve. The second author is grateful to the Università degli Studi di Milano and the Università degli Studi di Palermo for their kind hospitality during his stays in Milan and in Palermo.


\providecommand{\bysame}{\leavevmode\hbox to3em{\hrulefill}\thinspace}
\providecommand{\MR}[1]{}
\providecommand{\MRhref}[2]{%
  \href{http://www.ams.org/mathscinet-getitem?mr=#1}{#2}
}
\providecommand{\href}[2]{#2}

\end{document}